\documentclass[reqno,11pt]{amsart}

\usepackage{style}

\begin{document}

\title{Schur-Positivity of Short Chords in Matchings}
\author{Avichai Marmor}
\address{Bar-Ilan University}
\email{avichai@elmar.co.il}
\date{}
\thanks{Partially supported by the Israel Science Foundation, Grant No.\ 1970/18, and by the European Research Council under the ERC starting grant agreement No.\ 757731 (LightCrypt).}

\begin{abstract}

We prove that the set of matchings with a fixed number of unmatched vertices is Schur-positive with respect to the set of short chords. 
Two proofs are presented. 
The first proof applies a new combinatorial criterion for Schur-positivity,  
while the second is bijective. 
The coefficients in the Schur expansion are derived, 
and interpreted in terms of Bessel polynomials. 
Then, we present a variant of Knuth equivalence for matchings, and show that every equivalence class corresponds to a Schur function. We proceed to find various refined Schur-positive sets, including the set of matchings with a prescribed crossing number and the set of matchings with a given number of pairs of intersecting chords. 
Finally, we
characterize all the matchings $m$ such that the set of matchings avoiding $m$ is Schur-positive.
\end{abstract}

	
\maketitle
	
\tableofcontents

\section{Introduction}
\label{sec:introduction}

Denote the set of nonnegative integers by $\NN$, and the set of positive integers by $\PP$. Given $N\in \NN$ and a set $\A$, together with a set-valued function $D:\A\to 2^{[N-1]}$ which is sometimes called a \emph{statistic}, we define its quasisymmetric generating function
\[
    \Q_{D} (\A) = \sum_{a \in \A} F_{N, D(a)},
\]
where $F_{N, S}$ is the \emph{fundamental quasisymmetric function} introduced by Gessel~\cite{knuth_classes_schur_positive}, indexed by a subset $S \subseteq [N-1]$ (see Section~\ref{subsec:definitions symmetric functions} for more background). We write $\Q(\A)$ instead of $\Q_{D}(\A)$ when $D$ is clear from the context.
The following long-standing problem was first addressed in~\cite{gessel1993counting}.
\begin{problem}[Gessel and Reutenauer]
    Find sets $\A$ and statistics $D$ for which $\Q_{D}(\A)$ is symmetric.
\end{problem}

When $\Q_{D}(\A)$ is symmetric, we say that $\A$ is symmetric with respect to the statistic $D$ (we omit $D$ when it is clear from the context). If $\A$ is symmetric, then there is a unique expansion
\[
    \Q_{D} (\A) = \sum_{\lambda \vdash N} c_\lambda s_\lambda,
\]
where $c_\lambda \in \QQ$ and $s_\lambda$ are Schur functions (discussed in Section~\ref{subsec:definitions symmetric functions}). The coefficients $c_\lambda$ are called the \emph{Schur coefficients} of $\A$.
If $c_\lambda \in \NN$ for all $\lambda \vdash N$, we say that $\A$ is Schur-positive with respect to $D$.

Gessel and Reutenauer also addressed the following problem:
\begin{problem}
    Find Schur-positive sets $\A$.
\end{problem}

The study of symmetric functions is partially motivated by the correspondence between symmetric functions of degree $N$ and class functions on the symmetric group $S_N$, which is established through the Frobenius characteristic map described in~\cite[Chapter 4.7]{sagan_book}. A symmetric function is Schur-positive if its corresponding class function is a proper character (see~\cite{adin_roichman_fine_set} for a detailed explanation). This perspective provides a rich algebraic framework for studying symmetric and Schur-positive sets.

During the last few decades, many Schur-positive subsets of the symmetric group $S_N$ were constructed with respect to the standard descent function of $S_N$, where a permutation $\pi \in S_N$ has $i \in \Des(\pi)$ if and only if $\pi(i) > \pi(i+1)$. A fundamental construction is due to Gessel~\cite{knuth_classes_schur_positive}, who proved that every set of permutations which is closed under the Knuth equivalence relations is Schur-positive. Furthermore, he showed that every Knuth class $\A \subset S_N$ (i.e., equivalence class of the Knuth relations) satisfies $\Q_{\Des}(\A) = s_\lambda$ for some $\lambda \vdash N$. This result has many direct consequences. For example, for every $J \subseteq [N-1]$, the set $D_{N,J}^{-1} = \{ \pi^{-1} \mid \Des(\pi)=J\} \subseteq S_N$ is Schur-positive because it is closed under the Knuth relations.

Other examples of Schur-positive sets of permutations include all conjugacy classes, as proven by Gessel and Reutenauer~\cite{gessel1993counting}, and the set of permutations $\pi \in S_N$ with a fixed number of inversions (i.e., $i<j$ such that $\pi(i) > \pi(j)$), a result demonstrated by Adin and Roichman \cite[Prop.\ 9.5]{adin_roichman_fine_set}.

Some Schur-positive sets that do not consist of permutations were found as well. For instance, Gessel~\cite{knuth_classes_schur_positive} proved that for every partition $\lambda \vdash N$, the set $\Syt(\lambda)$ of standard Young tableaux of shape $\lambda$ is Schur-positive (see Section~\ref{subsec:schur positive sets} for details).

In this work, we focus on the set of matchings of a given set of vertices.
\begin{definition}
\label{def:matching}
    A matching $m$ on a finite set of vertices $S \subseteq \PP$ is an unordered partition of $S$ into blocks, each block of size $1$ or $2$. If a block of a matching $m$ contains only the vertex $i$, we say that $i$ is an \emph{unmatched vertex} of $m$ and denote $(i) \in m$. If a block contains both $i$ and $j$ for $i<j$, we say that $(i,j)$ is a \emph{chord} of $m$ and denote $(i,j) \in m$. That is, the notation $(i,j)\in m$ implies that $i<j$. When $(i,j) \in m$, we also say that $i$ opens this chord and $j$ closes it. If a matching $m$ has no unmatched vertex, then we say that $m$ is a perfect matching.
\end{definition}

We will focus on matchings on the set $[N] := \{1,\dots, N\}$ for some $N \in \NN$. Denote by $\M_N$ the set of matchings on $[N]$, and by $\M_{N,f}$ the set of matchings on $[N]$ with $f$ unmatched vertices.

For example, 
\[
    \M_{4,0} = \big\{\{(1,2),\; (3,4)\},\ \{(1,3),\; (2,4)\},\ \{(1,4),\; (2,3)\}\big\}
\]
and
\[ 
    \M_{3,1} = \big\{\{(1,2),\; (3)\},\ \{(1,3),\; (2)\},\ \{(1),\; (2,3)\}\big\}.
\]

\begin{definition}
\label{def:intersect}
    Let $m$ be a matching on $S$. If $(i_1, i_3),\: (i_2,i_4) \in m$ for $i_1 < i_2 < i_3 < i_4$ then we say that the chords $(i_1, i_3)$ and $(i_2, i_4)$ \emph{intersect}.
\end{definition}

\begin{definition}
\label{def:short}
    Let $m \in \M_N$ be a matching on $[N]$. If $(i,i+1) \in m$, then it is called a \emph{short chord}. We denote $\Short(m) := \{i \in [N-1] \mid (i,i+1) \in m\}$.

\end{definition}

For example, consider the matching $m = \{(1,3),\; (2,6),\; (4,5)\}$ as in Figure~\ref{fig:matching example introduction}. In this matching the chords $(1,3)$ and $(2,6)$ intersect, and $\Short(m) = \{4\}$ (since $(4,5) \in m$).
\begin{figure}[htb]
    \begin{center}
        \begin{tikzpicture}[node distance={15mm}, thick, main/.style = {draw, circle, scale=1.2}]  
            \node[main] (1) {$1$};
            \node[main] (2) [right of=1] {$2$};
            \node[main] (3) [right of=2] {$3$};
            \node[main] (4) [right of=3] {$4$};
            \node[main] (5) [right of=4] {$5$};
            \node[main] (6) [right of=5] {$6$};

            \draw (1) to [out=67,in=113,looseness=0.8] (3);
            \draw (2) to [out=55,in=125,looseness=0.65] (6);
            \draw (4) to [out=67,in=113,looseness=1] (5);
        \end{tikzpicture}
    \end{center}
    \caption{The matching $m = \{(1,3),\; (2,6),\; (4,5)\}$.}
    \label{fig:matching example introduction}
\end{figure}

Note that $N \notin \Short(m)$ for every $m \in \M_N$. This holds even where $(1,N) \in m$.

Enumerative properties of short chords in perfect matchings were studied extensively. The distribution of the number of short chords in perfect matchings appears in entry A079267 of the OEIS~\cite{oeis}, and was analyzed for example by Cameron and Killpatrick~\cite{killpatrick2020statistics}. McSorley and Feinsilver~\cite{bessel_polynomials_mathcings} discovered that short chords of matchings are closely related to Bessel polynomials, as will be discussed in Section~\ref{subsec:schur coefficients bessel}.

The study of short chords of matchings is also motivated by the involutive length and its corresponding poset studied by Adin, Postnikov and Roichman~\cite{gelfand_involutions}. Related posets derived from the Bruhat order were studied by Richardson and Springer~\cite{richardson1990bruhat}, Hultman~\cite{hultman2008twisted}, Deodhar and Srinivasan~\cite{deodhar2001statistic} and others, motivated by the topology of linear algebraic groups (see~\cite{hultman2008twisted} for details). Further discussion on the algebraic motivations can be found in Section~\ref{subsec:schreier graph of perfect matchings}.




We prove the following:
\begin{theorem}
\label{thm:matchings schur positive}
    Let $n,f \in \NN$ be nonnegative integers, and denote $N=2n+f$. Then the set $\M_{N,f}$ is Schur-positive with respect to $\Short$. Furthermore, its Schur expansion is given by the following formula:
    \[
        \Q_{\Short} (\M_{N,f}) = \sum_{k=0}^n |\{ m\in \M_{N-2k,f} \mid \Short(m) = \emptyset\}|\ s_{N-k,k}.
    \]
\end{theorem}

It turns out  that the Schur coefficients of $\Q_{\Short} (\M_{N,f})$ 
may be explicitly interpreted in terms of Bessel polynomials, see 
Corollary~\ref{cor:schur expansion bessel} 
below.

We give two proofs of Theorem~\ref{thm:matchings schur positive}. The first proof relies on a new criterion for Schur-positivity of sparse statistics:

\begin{definition}
\label{def:sparse}
    We say that a set $J \subseteq [N-1]$ is \emph{sparse} if $\{j,j+1\} \nsubseteq J$ for every $1 \le j \le N-2$.

    For a set $\A$ and a function $D:\A \to 2^{[N-1]}$, we say that $D$ is \emph{sparse} if $D(a)$ is sparse for all $a \in \A$.
\end{definition}

\begin{theorem}
\label{thm:criterion schur positive two rows}
    Let $\A$ be a finite set with a statistic $D:\A \to 2^{[N-1]}$, and denote $n = \lfloor\frac{N}{2} \rfloor$. Then the following statements are equivalent:
    \begin{itemize}
        \item $D$ is sparse, and for every sparse $J \subseteq [N-1]$, the cardinality of the set $\{a \in \A \mid D(a) \supseteq J\}$ depends on the size of $J$ only.
        \item $\A$ is symmetric with respect to $D$, with a Schur expansion of the form $\Q_{D}(\A) = \sum_{k=0}^n c_k s_{N-k,k}$ for some $c_k \in \ZZ$.
    \end{itemize}

    Furthermore, if these statements hold then $\A$ is Schur-positive, and its Schur expansion is
    \[
        \Q_{D} (\A) = \sum_{k=0}^n \left|\left\{ a\in \A \mid D(a) = \{1,3,5,\dots, 2k-1\} \right\}\right|\: s_{N-k,k}.
    \]
    
\end{theorem}

This new criterion implies Theorem~\ref{thm:matchings schur positive},
as shown at the beginning of Section~\ref{sec:short chords in matchings}, and can be applied to many other sets as well.
It is also applied in a subsequent paper \cite{gallai} to prove Schur-positivity of Gallai colorings and transitive colorings of complete graphs.

The second proof of Theorem~\ref{thm:matchings schur positive} provides a bijection from the set of matchings to a multiset of SYTs,
and applies the bijective criterion of Adin and Roichman for Schur-positivity \cite[Prop.\ 9.1]{adin_roichman_fine_set}, presented in Theorem~\ref{thm:criterion schur positive young}.

Following the bijective proof, we present an equivalence relation on matchings, motivated by Knuth equivalence of permutations~\cite{knuth1970permutations}:
\begin{definition}\label{def:knuth like equivalence}
    We say that two matchings $m_1,m_2 \in \M_N$ are \emph{Knuth equivalent} if one can be obtained from the other by a sequence of \emph{elementary Knuth transformations}:
    \begin{itemize}
        \item Replace the chords $(i,i+1),\; (i+2)$ with the chords $(i),\; (i+1,i+2)$ or vice versa (i.e.,\ interchange a short chord with an adjacent unmatched vertex).
        \item Replace the chords $(i,i+1),\; (i+2,j)$ with the chords $(i,j),\; (i+1,i+2)$, or $(i,i+1),\; (j,i+2)$ with $(j,i),\; (i+1,i+2)$, or vice versa (i.e.,\ interchange a short chord with an adjacent endpoint of another chord).
    \end{itemize}
\end{definition}

We also define the core of a given matching:

\begin{definition}[Definition~\ref{def:core of matching} below]
    The \emph{core} of a given matching $m \in \M_N$, denoted $\core(m)$, is obtained by repeatedly removing short chords from the matching until no short chords remain. The remaining vertices are then re-indexed with natural numbers starting from $1$ while preserving their relative order.

\end{definition}

A classical result due to Knuth~\cite{knuth1970permutations} states that two permutations $\pi_1,\pi_2 \in S_N$ are Knuth equivalent if and only if $P(\pi_1) = P(\pi_2)$, where $P(\pi)$ denotes the insertion tableau of a permutation $\pi$ defined by the Robinson-Schensted correspondence (see \cite[Section 3]{sagan_book} for details). We prove the following analogous result for matchings:
\begin{theorem}[Theorem~\ref{thm:knuth like equivalent iff same core} below]
    Two matchings $m_1, m_2 \in \M_N$ are Knuth equivalent if and only if $\core(m_1) = \core(m_2)$.
\end{theorem}

We show that Knuth classes of matchings, similarly to Knuth classes of permutations, have Schur functions as their generating functions:

\begin{theorem}[Corollary~\ref{cor:closed under knuth like is schur positive} below]
    Every set $\M \subseteq \M_N$ of matchings that is closed under Knuth equivalence is Schur-positive with respect to $\Short$. Moreover, if $\M$ is a Knuth equivalence class, then its generating function is $\Q_{\Short}(\M) = s_{N-k,k}$, where $N-2k$ is the number of vertices of $\core(m)$ for some arbitrary matching $m \in \M$.
\end{theorem}

Furthermore, utilizing Knuth classes, we explore in Section~\ref{sec:refinements} various refined Schur-positive sets, including:
\begin{itemize}
    \item The set of $k$-crossing matchings (i.e.,\ where $k$ is the maximal cardinality of a set of pairwise intersecting chords).
    \item The set of matchings with exactly $k$ pairs of intersecting chords.
\end{itemize}

Lastly, we define $\M_{N,f}(m)$ as the set of matchings in $\M_{N,f}$ that avoid the pattern $m$. In Proposition~\ref{prop:characterize pattern avoiding schur positive singletons}, we provide a characterization of the matchings $m$ for which $\M_{N,f}(m)$ is Schur-positive for all $N$ and $f$.

The remainder of this work is organized as follows: Section~\ref{sec:preliminaries} provides necessary background. In Section~\ref{sec:criterion two rows}, we prove a necessary and sufficient criterion for sparse Schur-positivity (Theorem~\ref{thm:criterion schur positive two rows}). In Section~\ref{sec:short chords in matchings}, we focus on matchings. We first derive the Schur-positivity of short chords in matchings (Theorem~\ref{thm:matchings schur positive}) from the above criterion. Then, we present an alternative bijective proof, followed by a description of the Schur coefficients of matchings in terms of Bessel polynomials. In Section~\ref{sec:refinements}, we utilize the bijection to refine the Schur-positivity property of Theorem~\ref{thm:matchings schur positive}. Finally, Section~\ref{sec:further research} concludes with further remarks and open problems.

\section{Preliminaries and notation}
\label{sec:preliminaries}

\begin{definition}
    For nonnegative integers $i,j$, we define the \emph{interval}
    \[
        [i,j] :=
        \begin{cases}
            \{i,i+1,\dots,j\} & \text{if $i \le j$,}\\
            \emptyset & \text{otherwise.}
        \end{cases}
    \]
    We also define $[i] = [1,i]$.
\end{definition}

\begin{definition}
\label{def:elements with given stat}
    Given a set $\A$, a function $D:\A \to 2^{[N-1]}$ and a set $J \subseteq [N-1]$, we denote $\A(D=J) := \{a \in \A \mid D(a) = J\}$ and $\A(D \supseteq J) := \{a \in \A \mid D(a) \supseteq J\}$.
\end{definition}

\subsection{Symmetric and quasisymmetric functions}
\label{subsec:definitions symmetric functions}

\begin{definition}
    Given $N \in \NN$, a \emph{partition} $\lambda$ of $N$ (denoted $\lambda \vdash N$), is a weakly-decreasing sequence $\lambda = (\lambda_1 \ge \lambda_2 \ge \dots \ge \lambda_\ell)$ of positive integers, such that $\lambda_1 + \dots + \lambda_\ell = N$. The values $(\lambda_1,\dots,\lambda_\ell)$ are called the \emph{parts} of $\lambda$.
    
    A \emph{composition} $\alpha$ of $N$ (denoted $\alpha \vDash N$) is a sequence $\alpha = (\alpha_1, \alpha_2, \dots, \alpha_\ell )$ of positive integers, such that $\alpha_1 + \dots + \alpha_\ell = N$. Every partition is a composition.
\end{definition}

Compositions of $N$ are in bijection with subsets of $[N-1] := \{1, \dots, N-1\}$ by 
\[
    [N-1] \supseteq \{i_1, i_2, \dots, i_\ell \} \mapsto (i_1, i_2 - i_1, \dots, i_\ell - i_{\ell-1}, N - i_\ell) \vDash N.
\]
For example, for $N = 10$, $\{2, 4, 7, 8\} \mapsto (2, 2, 3, 1, 2)$. The composition associated to a set $S$ is denoted $\alpha_S$, and the set associated to a composition $\alpha$ is denoted $S_\alpha$.

We say that two compositions $\alpha, \beta \vDash N$ are \emph{equivalent} (denoted $\alpha \sim \beta$), if $\beta$ is a rearrangement of the entries of $\alpha$.

Denote the ring of symmetric functions by $\sym$ and the ring of quasisymmetric functions by $\qsym$ (see, for example, \cite[Section 1]{S3_patterns_list} for details). The space of homogeneous symmetric functions of degree $N$ is denoted by $\sym_N$, and the space of homogeneous quasisymmetric functions of degree $N$ is denoted by $\qsym_N$.

The space $\qsym_N$ has several important bases. In this work, we focus on the \emph{fundamental basis}, consisting of the functions
\[
    F_{N,S} := \sum_{\substack{i_1 \le i_2 \le \dots \le i_N \\ \forall j \in S:i_j < i_{j+1}}} x_{i_1} x_{i_2} \dots x_{i_N},\quad S \subseteq [N - 1].
\]
We write $F_S$ instead of $F_{N,S}$ when $N$ is clear from the context.

The space $\sym_N$ has several standard bases as well. In this work, we focus on the Schur basis, which consists of the Schur functions $s_{\lambda}$, where $\lambda$ is a partition of $N$. 
The definition and properties of Schur functions can be found in \cite[Section 4.4]{sagan_book}. Here, we adopt the combinatorial approach to Schur functions, as described in Theorem~\ref{thm:schur functions young tableaux} and Theorem~\ref{thm:criterion schur positive young} below.

\subsection{Symmetric and Schur-positive sets}
\label{subsec:schur positive sets}
As mentioned in Section~\ref{sec:introduction}, a set $\A$ is symmetric with respect to a statistic $D:\A \to 2^{[N-1]}$ if the generating function $\Q_{D}(\A)$ is a symmetric function. Moreover, it is Schur-positive if all Schur coefficients are nonnegative integers.


One of the fundamental constructions of Schur-positive sets, regarding sets of standard Young tableaux (SYT), is due to Gessel~\cite{knuth_classes_schur_positive}. Let $\Syt(\lambda)$ denote the set of standard Young tableaux of shape $\lambda$. We draw tableaux in English notation, as in Figure~\ref{fig:SYT example}. The \emph{descent set} of $T\in \Syt(\lambda)$ is
\[
    \Des(T) := \{i \in [N-1] \mid i+1 \text{ appears in a lower row than $i$ in $T$}\}.
\]
For example, the descent set of the SYT in Figure~\ref{fig:SYT example} is $\{2,4,7,8\}$.

\begin{figure}[htb]
	\[
            \young(1247,36,58,9)
        \]
	\caption{A SYT of shape $\lambda = (4,2,2,1)$.} 
	\label{fig:SYT example}
\end{figure}

The entry in row $i$ and column $j$ of a tableau $T \in \Syt(\lambda)$ is denoted by $T_{i,j}$. In addition, we define $\row_i(T) := \{T_{i,j} \mid 1 \le j \le \lambda_i\}$ as the set of entries in the $i$-th row of $T$. For example, if we consider the SYT shown in Figure~\ref{fig:SYT example}, then $T_{3,2} = 8$ and $\row_3(T) = \{5,8\}$.

\begin{theorem}[Gessel~\cite{knuth_classes_schur_positive}]
\label{thm:schur functions young tableaux}
    For every $\lambda \vdash N$, the set $\Syt(\lambda)$ is Schur-positive with respect to $\Des$. Moreover, $\Q(\Syt(\lambda)) = s_\lambda$.
\end{theorem}

In 2015, Adin and Roichman proved the following criterion.
\begin{theorem}[{\cite[Prop.\ 9.1]{adin_roichman_fine_set}}]
\label{thm:criterion schur positive young}
    A set $\A$ is symmetric with respect to $D:\A\to 2^{[N-1]}$ if and only if
    \[
        \sum_{a \in \A} \boldsymbol{t}^{D(a)} = \sum_{\lambda \vdash N} c_\lambda \sum_{T \in \Syt(\lambda)} \boldsymbol{t}^{\Des(T)}
    \]
    for some values $c_\lambda$, where $\boldsymbol{t}^J := \prod_{j \in J} t_j$ for $J \subseteq [N-1]$. The coefficients $c_\lambda$ are the Schur-coefficients of $\A$. Moreover, $\A$ is Schur-positive if and only if $c_{\lambda} \in \NN$ for all $\lambda \vdash N$.
\end{theorem}

This criterion implies that proving the Schur-positivity of a set is achievable by establishing a statistic-preserving bijection between the set and SYTs of shapes corresponding to a specific multiset.

In this paper, we will also apply a recently formulated criterion for symmetry~\cite{my_pattern_theorem}.
\begin{definition}
\label{def:respect composition}
    Let $\A$ be a finite set with a statistic $D:\A \to 2^{[N-1]}$. The set of elements that \emph{respect} a given composition $\alpha \vDash N$, denoted $\A_D(\alpha)$, consists of the elements $a \in \A$ such that $D(a) \subseteq S_\alpha$, where $S_\alpha$ is the set corresponding to the composition $\alpha$. When $D$ is clear from the context, we may write $\A(\alpha)$ instead.
\end{definition}

\begin{lemma}
\label{lem:my symmetry criterion}
    A set $\A$ is symmetric if and only if $|\A(\alpha)| = |\A(\beta)|$ for all $\alpha \sim \beta \vDash N$.
\end{lemma}

Note that only sets of permutations are considered in~\cite{my_pattern_theorem}. However, Lemma~\ref{lem:my symmetry criterion} applies to other sets as well.


    

    

Another useful result about symmetric sets and symmetric functions is due to Bloom and Sagan:
\begin{lemma}[Bloom and Sagan {\cite[Lemma 2.2]{BloomSagan}}]
\label{lem:singleton not symmetric}
    
    For every set $S \subseteq [N-1]$, the function $F_{N,S}$ is symmetric if and only if $S=[N-1]$ or $S=\emptyset$.
\end{lemma}


\section{Proof of Theorem~\ref{thm:criterion schur positive two rows}}
\label{sec:criterion two rows}

Denote $\odds(k) := \{1,3,5,\dots, 2k-1\}$ for a nonnegative integer $k$. In particular, $\odds(0) = \emptyset$.

Recall Definition~\ref{def:sparse} and Definition~\ref{def:elements with given stat}. We divide the assertions of Theorem~\ref{thm:criterion schur positive two rows} into two propositions:
\begin{lemma}
\label{lem:if schur positive two rows then sparse and good sizes}
    Let $\A$ be a symmetric set with respect to a statistic $D:\A \to 2^{[N-1]}$, and denote $n = \lfloor\frac{N}{2} \rfloor$. Assume that the Schur expansion of $\A$ is $\Q_{D}(\A) = \sum_{k=0}^n c_k s_{N-k,k}$ for some $c_k \in \ZZ$. Then
    $D$ is a sparse function, and for every sparse set $J \subseteq [N-1]$, the cardinality of the set $\A(D \supseteq J)$ depends on the size of $J$ only.
\end{lemma}
\begin{lemma}
\label{lem:if sparse and good sizes then schur positive two rows}
    Let $\A$ be a finite set with a sparse statistic $D:\A \to 2^{[N-1]}$, and denote $n = \lfloor\frac{N}{2} \rfloor$. Assume that for every sparse set $J \subseteq [N-1]$, the cardinality of the set $\A(D \supseteq J)$ depends on the size of $J$ only.
    Then $\A$ is Schur-positive, and its Schur expansion is
    \[
        \Q_{D} (\A) = \sum_{k=0}^n |\A(D=\odds(k))|\; s_{N-k,k}.
    \]
\end{lemma}

We prove Lemma~\ref{lem:if schur positive two rows then sparse and good sizes} in Section~\ref{subsec:proof lemma if schur positive two rows then sparse and good sizes}, and then show two proofs for Lemma~\ref{lem:if sparse and good sizes then schur positive two rows}. The first proof (presented in Section~\ref{subsec:proof if sparse and good sizes then schur positive two rows by induction}) is inductive. The second proof (presented in Section~\ref{subsec:proof if sparse and good sizes then schur positive two rows by superstandard}) is more involved, and it demonstrates the power of column superstandard tableaux, introduced by Hamaker, Pawlowski and Sagan~\cite{S3_patterns_list}, in proving Schur-positivity properties. We believe this approach can be applied in other cases as well.

Before proving these lemmas, let us prove a simple folklore lemma that will be useful for both lemmas:
\begin{lemma}
\label{lem:symmetric complement is symmetric}
    Let $\A$ be a symmetric set with respect to a statistic $D:\A \to 2^{[N-1]}$. Then $\A$ is also symmetric with respect to the complementary statistic $\bar{D}:\A \to 2^{[N-1]}$ defined by $a \mapsto [N-1]\setminus D(a)$.
\end{lemma}
\begin{proof}
    The set $\A$ is symmetric with respect to $D$, so by Theorem~\ref{thm:criterion schur positive young} there exist values $c_{\lambda},\ \lambda \vdash N$ such that
    \[
        \sum_{a \in \A} \boldsymbol{t}^{D(a)} = \sum_{\lambda \vdash N} c_\lambda \sum_{T \in \Syt(\lambda)} \boldsymbol{t}^{\Des(T)},
    \]
    where $\boldsymbol{t}^J := \prod_{j \in J} t_j$ for $J \subseteq [N-1]$, and therefore
    \[
        \sum_{a \in \A} \boldsymbol{t}^{[N-1]\setminus D(a)} = \sum_{\lambda \vdash N} c_\lambda \sum_{T \in \Syt(\lambda)} \boldsymbol{t}^{[N-1]\setminus \Des(T)} = \sum_{\lambda \vdash N} c_\lambda \sum_{T \in \Syt(\lambda')} \boldsymbol{t}^{\Des(T)},
    \]
    where $\lambda'$ is the partition conjugate to $\lambda$.
\end{proof}

\subsection{Proof of Lemma~\ref{lem:if schur positive two rows then sparse and good sizes}}
\label{subsec:proof lemma if schur positive two rows then sparse and good sizes}
\begin{proof}[\unskip\nopunct]
    The set $\A$ is symmetric, and its Schur expansion is $\Q_{D}(\A) = \sum_{k=0}^n c_k s_{N-k,k}$. Therefore, by Theorem~\ref{thm:criterion schur positive young}, we have
    \begin{equation}\label{eqn:proof if schur positive two rows then sparse}
        \sum_{a \in \A} \boldsymbol{t}^{D(a)} = \sum_{k=0}^n c_k \sum_{T \in \Syt(N-k,k)} \boldsymbol{t}^{\Des(T)}.
    \end{equation}
    First of all, notice that the set $\Des(T)$ is sparse for all $T \in \Syt(N-k,k)$. Therefore, by Equation~\eqref{eqn:proof if schur positive two rows then sparse}, the set $D(a)$ is sparse for all $a \in \A$.
    
    It remains to show that for every sparse set $J \subseteq [N-1]$ of size $k$, we have $\left|\A(D \supseteq J)\right| = \left|\A(D \supseteq \odds(k))\right|$. Denote $\alpha = (2^k,1^{N-2k}) \vDash N$. Additionally, let $i_1 < \dots < i_{N-k-1} \in [N-1]\setminus J$ denote the elements not in $J$. Since $J$ is sparse, the differences $i_{t+1} - i_t$ for $1 \le t \le N-k-2$, as well as $i_1$ and $N - i_{N-k-1}$, are either 1 or 2. Those that are equal to 2 correspond to the elements of $J$. The composition associated to the set $[N-1]\setminus J$ is denoted by $\beta := (i_1, i_2 - i_1, i_3 -i_2, \dots, i_{N-k-1} - i_{N-k-2}, N-i_{N-k-1}) \vDash N$. By Lemma~\ref{lem:symmetric complement is symmetric}, $\A$ is symmetric with respect to the complementary statistic $\bar{D}$. Notably, the compositions $\alpha$ and $\beta$ are equivalent, as they both have $k$ occurrences of $2$ and $N-2k$ occurrences of $1$. By Lemma~\ref{lem:my symmetry criterion}, we obtain that $|\A_{\bar{D}}(\beta)| = |\A_{\bar{D}}(\alpha)|$. By Definition~\ref{def:respect composition}, we obtain that $\A_{\bar{D}}(\beta) = \A(D \supseteq J)$ and $\A_{\bar{D}}(\alpha) = \A(D \supseteq \odds(k))$. Consequently, $|\A(D \supseteq J)| = |\A(D \supseteq \odds(k))|$, as required.
\end{proof}

\subsection{First proof of Lemma~\ref{lem:if sparse and good sizes then schur positive two rows}}
\label{subsec:proof if sparse and good sizes then schur positive two rows by induction}
We start by introducing an observation, which will serve as a crucial step in our inductive argument.
\begin{lemma}
\label{lem:remove first descent from syt}
    Let $k_1 \ge k_2$ be two nonnegative integers, and let $J \subseteq [k_1 + k_2 - 1]$ be a set. Then
    \[
        \left|\Syt(k_1, k_2)(\Des \supseteq J)\right| = \left|\Syt\big(k_1+1, k_2+1\big)\big(\Des \supseteq (\{1\}\cup (J+2))\big)\right|,
    \]
    where $J+2 = \{j + 2 \mid j \in J\}$.
\end{lemma}
\begin{proof}
    The proof is bijective: We associate a given tableau $T \in \Syt(k_1, k_2)$ such that $\Des(T) \supseteq J$ with the tableau $T' \in \Syt(k_1+1, k_2+1)$ defined by $T'_{1,1} = 1,\; T'_{2,1} = 2$, and for all $i \in \{1,2\}$ and $j \in [k_i]$, $T'_{i,j+1} = T_{i,j}+2$ (recall that $T_{i,j}$ is the entry in row $i$ and column $j$ of $T$). Intuitively, we first increase every entry of $T$ by $2$. Then, we insert a new column with the letters $1,2$ on the left of $T$, shifting the other columns to the right. Clearly, this process is injective, and $T'$ satisfies $T' \in \Syt(k_1+1, k_2+1)$ and $\Des(T) \supseteq \{1\}\cup (J+2)$. Furthermore, any $T' \in \Syt(k_1+1, k_2+1)$ with $\Des(T') \supseteq \{1\}\cup (J+2)$ can be obtained through this process.
\end{proof}

Now we are ready to prove Lemma~\ref{lem:if sparse and good sizes then schur positive two rows}.
\begin{proof}[First proof of Lemma~\ref{lem:if sparse and good sizes then schur positive two rows}]
    Let $\A$ be a finite set with a sparse statistic $D:\A \to 2^{[N-1]}$, and assume that for every sparse set $J \subseteq [N-1]$, the cardinality of the set $\A(D \supseteq J)$ depends on the size of $J$ only.
    By Theorem~\ref{thm:criterion schur positive young}, it suffices to show that
    \begin{equation}\label{eqn:first proof no consecutive indices is schur positive - generating function}
        \sum_{a \in \A} \boldsymbol{t}^{D(a)} = \sum_{k=0}^n |\A(D=\odds(k))| \sum_{T \in \Syt(N-k,k)} \boldsymbol{t}^{\Des(T)},
    \end{equation}
    where $n = \lfloor\frac{N}{2}\rfloor$.
    
    We prove it by induction on $N \in \NN$:
    
    For $N \le 1$, the statement holds trivially.
    
    Next, we assume that the statement holds for every value smaller than $N$ and proceed to prove it for $N$. Our goal is to establish the equation:
    \begin{equation}\label{eqn:first proof no consecutive indices is schur positive - set equality}
        \left|\A(D=J)\right| = \sum_{k=0}^n c_k \left|\Syt(N-k,k)(\Des=J)\right|
    \end{equation}
    for every set $J \subseteq [N-1]$, where $c_k = \left|\A(D=\odds(k))\right|$.
    
    We begin by considering the case $J = \emptyset$. In this case, the only SYT $T$ with $\Des(T) = \emptyset$ is the unique tableau of one row. Consequently, Equation~\eqref{eqn:first proof no consecutive indices is schur positive - set equality} simplifies to $\left|\A(D=\emptyset)\right| = \left|\A(D=\odds(0))\right|$, which obviously holds.
    
    It now remains to prove Equation~\eqref{eqn:first proof no consecutive indices is schur positive - set equality} for all $J \ne \emptyset$. By  the inclusion–exclusion principle, it suffices to prove that the following holds for all $J \ne \emptyset$:
    \begin{equation}\label{eqn:first proof no consecutive indices is schur positive - set containment}
        \left|\A(D \supseteq J)\right| = \sum_{k=0}^n c_k \left|\Syt(N-k,k)(\Des \supseteq J)\right|.
    \end{equation}
    
    The function $D$ is assumed to be sparse, and the set $\Des(T)$ is sparse for all $T \in \Syt(N-k,k)$. Therefore, if $J$ is not sparse then Equation~\eqref{eqn:first proof no consecutive indices is schur positive - set containment} holds.

    Next, let $J \ne \emptyset$ be a sparse set, and denote $|J| = i$. The set $\odds(i)$ is sparse too, and has $i$ elements. Therefore, by the assumptions of the lemma, we obtain that $\left|\A(D \supseteq J)\right| = \left|\A(D \supseteq \odds(i))\right|$.
    In addition, by Theorem~\ref{thm:schur functions young tableaux}, the set $\Syt(N-k,k)$ is Schur-positive and has the Schur expansion $\Q(\Syt(N-k,k)) = s_{N-k,k}$. Therefore, we can apply Lemma~\ref{lem:if schur positive two rows then sparse and good sizes} and deduce that
    \[
        \left|\Syt(N-k,k)(\Des \supseteq J)\right| = \left|\Syt(N-k,k)(\Des \supseteq \odds(i))\right|.
    \]
    Therefore, Equation~\ref{eqn:first proof no consecutive indices is schur positive - set containment} reduces to
    \begin{equation}\label{eqn:first proof no consecutive indices is schur positive - odds set}
        \left|\A(D \supseteq \odds(i))\right| = \sum_{k=0}^n c_k \left|\Syt(N-k,k)(\Des \supseteq \odds(i))\right|,
    \end{equation}
    where $i > 0$. Thus, the proof will be completed when we establish Equation~\eqref{eqn:first proof no consecutive indices is schur positive - odds set}.
    
    Denote $\A_1 = \{a \in \A \mid 1\in D(a)\}$. Furthermore, define a new statistic $D_1:\A_1 \to 2^{[N-3]}$ by $D_1(a) = (D(a)\setminus\{1\})-2$. Notice that $\A_1$ satisfies the requirements of Lemma~\ref{lem:if sparse and good sizes then schur positive two rows} with respect to $D_1$, where $N$ is replaced by $N-2$. Therefore, we may assume, by the induction hypothesis, that Lemma~\ref{lem:if sparse and good sizes then schur positive two rows} holds for $\A_1$. Thus, we can apply Equation~\eqref{eqn:first proof no consecutive indices is schur positive - set containment} for $\A_1$, while replacing $c_k$ with $|\A_1(D_1 = \odds(k))| = |\A(D = \odds(k+1))| = c_{k+1}$, and obtain
    \begin{equation}\label{eqn:first proof no consecutive indices is schur positive - induction}
        \left|\A_1(D_1 \supseteq J)\right| = \sum_{k=0}^{n-1} c_{k+1} \left|\Syt(N-2-k,k)(\Des \supseteq J)\right|.
    \end{equation}
    
    If we substitute $J = \odds(i-1)$ into Equation~\ref{eqn:first proof no consecutive indices is schur positive - induction}, we obtain
    \[
        \left|\A_1(D_1 \supseteq \odds(i-1))\right| = 
         \sum_{k=0}^{n-1} c_{k+1} \left|\Syt(N-2-k,k)(\Des \supseteq \odds(i-1))\right|.
    \]

    Notice that $\left|\A_1(D_1 \supseteq \odds(i-1))\right| = \left|\A(D \supseteq \odds(i))\right|$. Moreover, by Lemma~\ref{lem:remove first descent from syt}, the number of tableaux $T$ of shape $(N-k-1,k-1)$ with $\Des(T) \supseteq \odds(i-1)$ is equal to the number of tableaux $T$ of shape $(N-k,k)$ with $\Des(T) \supseteq \odds(i)$. Therefore, we can reformulate the equation and obtain that
    \[
        \left|\A(D \supseteq \odds(i))\right| = 
         \sum_{k=1}^{n} c_k \left|\Syt(N-k,k)(\Des \supseteq \odds(i))\right|.
    \]
    The only difference between this equation and Equation~\eqref{eqn:first proof no consecutive indices is schur positive - odds set} is the omission of the summand corresponding to $k=0$. However, this summand evaluates to zero and does not affect the overall expression. Indeed, $\left|\Syt(N,0)(\Des \supseteq \odds(i))\right| = 0$, since the unique SYT of shape $(N)$ has no descents.
\end{proof}

\subsection{Second proof of Lemma~\ref{lem:if sparse and good sizes then schur positive two rows}}
\label{subsec:proof if sparse and good sizes then schur positive two rows by superstandard}
As a first step of this proof, we prove that a set satisfying the conditions of Lemma~\ref{lem:if sparse and good sizes then schur positive two rows} is symmetric.
\begin{lemma}
\label{lem:sufficent condition symmetry two rows}
    Let $\A$ be a finite set with a sparse statistic $D:\A \to 2^{[N-1]}$, and denote $n = \lfloor\frac{N}{2} \rfloor$. Assume that there are constants $b_k \in \NN$ for $0 \le k \le n$, such that for every sparse set $J \subseteq [N-1]$, $|\A(D \supseteq J)| = b_{|J|}$.
    Then $\A$ is symmetric with respect to $D$.
\end{lemma}
\begin{proof}
    Following Lemma~\ref{lem:symmetric complement is symmetric}, it suffices to prove that $\A$ is symmetric with respect to the complementary statistic $\bar{D}$. As we proceed to prove it by Lemma~\ref{lem:my symmetry criterion}, let us find $|\A_{\bar{D}}(\alpha)|$ for all $\alpha = (\alpha_1, \dots, \alpha_\ell) \vDash N$.
    
    If a composition $\alpha$ has $\max_i (\alpha_i) > 2$, 
    then there exists $j \in [N-2]$ such that $j,j+1 \notin S_\alpha$, where $S_\alpha$ is the set corresponding to $\alpha$. By Definition~\ref{def:respect composition}, for every element $a \in \A_{\bar{D}}(\alpha)$ we have $\bar{D}(a) \subseteq S_\alpha$, so $j, j + 1 \notin \bar{D}(a)$. Consequently, $j,j+1 \in D(a)$, and the set $D(a)$ is not sparse. However,
    $D$ is assumed to be a sparse function, so $\A_{\bar{D}}(\alpha) = \emptyset$.
    
    Now assume that $\max_i (\alpha_i) \le 2$, and denote by $k = |\{i \mid \alpha_i = 2\}|$ the number of occurrences of $2$ in $\alpha$. Thus, we have $\ell = N - k$.
    Consider the set $J = [N-1] \setminus S_\alpha$. Notably, the set $J$ is sparse and consists of $k$ elements.
    Let $a \in \A$ be an element. We have $a \in \A_{\bar{D}}(\alpha)$ if and only if $\bar{D}(a) \subseteq S_\alpha$, or equivalently, $D(a) \supseteq J$. Thus, $|\A_{\bar{D}}(\alpha)| = |\A(D \supseteq J)| = b_k$.
    
    To conclude, if $\alpha$ contains an element larger than $2$ then $\A_{\bar{D}}(\alpha) = \emptyset$. Otherwise, the size of $\A_{\bar{D}}(\alpha)$ depends only on the number of occurrences of $2$ in $\alpha$. Therefore, $|\A_{\bar{D}}(\alpha)| = |\A_{\bar{D}}(\beta)|$ for all $\alpha \sim \beta \vDash N$. By Lemma~\ref{lem:my symmetry criterion}, we obtain that $\A$ is symmetric.
\end{proof}

As the next step of the proof, we prove that $\A$ is Schur-positive and find its Schur coefficients. For this, we define a total order on partitions $\lambda \vdash N$:
\begin{definition}
\label{def:conjugate order}
    Let $\lambda, \mu \vdash N$ be partitions of $N$. Let $\lambda_i' = |\{j \mid \lambda_j \ge i\}|$ denote the length of the $i$-th column in the Young diagram of $\lambda$. We say that $\mu$ is \emph{larger} than $\lambda$ in the \emph{conjugate order}, and denote $\mu' > \lambda'$, if there exists $i$ such that $\mu_j' = \lambda_j'$ for all $j < i$ and $\mu_i' > \lambda_i'$.
\end{definition}

Following Hamaker, Pawlowski and Sagan \cite[Section 5]{S3_patterns_list}, we define the \emph{column superstandard} Young tableau of shape $\lambda \vdash N$, obtained by filling the columns of the Young diagram of shape $\lambda$ one by one. We denote it by $T_\lambda \in \Syt(\lambda)$.
Formally, $(T_\lambda)_{i,j} = \lambda_1' + \cdots + \lambda_{j-1}' + i$. For example, if $\lambda = (4,2,2,1)$, then $T_\lambda$ is the SYT in Figure~\ref{fig:SYT_lambda example}.

\begin{figure}[htb]
	\[
            \young(1589,26,37,4)
        \]
	\caption{The SYT $T_\lambda$ for $\lambda = (4,2,2,1)$.} 
	\label{fig:SYT_lambda example}
\end{figure}

The power of these notions may be reflected by the following statement:
\begin{lemma}
\label{lem:superstandard characterize coefficients}
    Let $\lambda, \mu \vdash N$ be two partitions. Then:
    \begin{enumerate}
        \item $\Des(T_\lambda) = [N-1] \setminus \{\lambda_{1}', \lambda_{1}'+\lambda_{2}', \dots\}$.
        \item If $\Des(T) = [N-1] \setminus \{\lambda_{1}', \lambda_{1}'+\lambda_{2}', \dots\}$ for some $T \in \Syt(\mu)$, then $\mu' \ge \lambda'$. Furthermore, if $\mu = \lambda$ then $T = T_\lambda$.
    \end{enumerate}
\end{lemma}

\begin{proof}
    The first assertion is obvious, so let us focus on the second assertion.
    
    Let $T \in \Syt(\mu)$ be a tableau, and assume that $\Des(T) = \Des(T_\lambda)$ and $T \ne T_\lambda$. We aim to show that $\mu' > \lambda'$. Let us denote the first column that differs between $T$ and $T_\lambda$ by $i$. The $i$-th column of $T_\lambda$ contains the entries $s+1,\dots,s+\lambda_{i}'$, where $s = \lambda_1' + \dots + \lambda_{i-1}'$. To prove $\mu' > \lambda'$, it suffices to show that all these entries also appear in the $i$-th column of $T$.
    
    Assume, by contradiction, that some of these entries appear in other columns of $T$. Let $x$ be the minimal such entry, and let $j \ne i$ be the column of $T$ that contains $x$. Since $s+1$ appears in the $i$-th column of $T$, we may assume that $x > s+1$. If $j < i$, this contradicts the assumption that $T$ and $T_\lambda$ agree in the first $i-1$ columns. On the other hand, if $j > i$, then $x$ appears in the first row of $T$, since it is the minimal entry of the $j$-th column of $T$. Consequently, we obtain that $x-1 \notin \Des(T)$, while $x-1 \in \Des(T_\lambda)$, which contradicts the assumption that $\Des(T) = \Des(T_\lambda)$.
\end{proof}

Lemma~\ref{lem:superstandard characterize coefficients} associates the Young diagram of shape $\lambda$ with the set $\Des(T_\lambda)$. As we will see later, this association is powerful in analysing the Schur coefficients of symmetric sets.


Standard Young tableaux with at most $2$ rows have a slightly stronger property:
\begin{lemma}
\label{lem:superstandard two rows characterize coefficients}
    Let $\lambda = (N-k_1,k_1)$ and $\mu = (N-k_2,k_2)$ be two partitions of $N$ with at most two parts each. Then:
    \begin{enumerate}
        \item $\Des(T_\lambda) = \odds(k_1)$.
        \item If $\Des(T) = \odds(k_1)$ for some $T \in \Syt(\mu)$, then $k_1 = k_2$ and $T = T_\lambda$.
    \end{enumerate}
\end{lemma}

    

\begin{proof}
    The first assertion is obvious, so let us focus on the second assertion.

    Let $T \in \Syt(\mu)$ be a tableau, and suppose that $\Des(T) = \Des(T_\lambda)$. Given that both $T$ and $T_\lambda$ have two rows each, it suffices to show that $\row_2(T) = \row_2(T_\lambda)$. Since $\Des(T) = \{1,3,\dots,2k_1-1\}$, it follows that $\{1,3,\dots,2k_1-1\} \subseteq \row_1(T)$ and $\{2,4,\dots,2k_1\} \subseteq \row_2(T)$. Consequently, we have $2k_1+1 \in \row_1(T)$. As $T$ has no descents beyond index $2k_1-1$, all entries greater than $2k_1$ must appear in the first row. Therefore, we conclude that $\row_2(T) = \{2,4,\dots,2k_1\}$.
\end{proof}

Now we are ready to prove that if a set is symmetric with respect to a sparse statistic then it is Schur-positive:
\begin{lemma}
\label{lem:if two rows and symmetric then schur positive}
    Let $\A$ be a symmetric set with respect to a sparse statistic $D:\A \to 2^{[N-1]}$, and denote $n = \lfloor\frac{N}{2} \rfloor$. Then $\A$ is Schur-positive, and its Schur expansion is
    \[
        \Q_{D} (\A) = \sum_{k=0}^n |\A(D = \odds(k))|\; s_{N-k,k}.
    \]
\end{lemma}
\begin{proof}
    The set $\A$ is assumed to be symmetric, so by Theorem~\ref{thm:criterion schur positive young}, we have
    \begin{equation}\label{eqn:proof symmetric set with no consecutive indices is schur positive}
        \sum_{a \in \A} \boldsymbol{t}^{D(a)} = \sum_{\lambda \vdash N} c_\lambda \sum_{T \in \Syt(\lambda)} \boldsymbol{t}^{\Des(T)},
    \end{equation}
    where $c_\lambda$ are the Schur coefficients of $\A$. It suffices to show that if there exists $k$ such that $\lambda=(N-k,k)$ then $c_\lambda = |\A(D = \odds(k))|$, and otherwise $c_\lambda=0$.

    If $\A = \emptyset$, then $\Q_{D}(\A) = 0$, and the statement holds. Therefore, we may assume that $\A \ne \emptyset$, and consequently, there exists $c_\lambda \ne 0$ for some partition $\lambda \vdash N$. Let $\mu \vdash N$ be a partition with $c_\mu \ne 0$, maximal in the conjugate order (i.e., such that $c_\lambda = 0$ for every $\lambda \vdash N$ with $\lambda' > \mu'$).
    
    Equation~\eqref{eqn:proof symmetric set with no consecutive indices is schur positive} implies that the equation
    \begin{equation}\label{eqn:proof symmetric set with no consecutive indices is schur positive - specific set}
        |\A(D = J)| = \sum_{\lambda \vdash N} c_\lambda \left|\Syt(\lambda)(\Des = J)\right|
    \end{equation}
    holds for all $J \subseteq [N-1]$. Let us find the right-hand side of Equation~\eqref{eqn:proof symmetric set with no consecutive indices is schur positive - specific set} when substituting $J = \Des(T_\mu)$. Let $T \in \Syt(\lambda)$ be a tableau with $\Des(T) = \Des(T_\mu)$. By Lemma~\ref{lem:superstandard characterize coefficients}, we have $T = T_\mu$ or $\lambda' > \mu'$. However, if $\lambda' > \mu'$ then $c_\lambda = 0$. Therefore, if $T \in \Syt(\lambda)$ has $\Des(T) = \Des(T_\mu)$ and $c_\lambda \ne 0$, then $T = T_\mu$. Consequently, substituting $J = \Des(T_\mu)$ into Equation~\eqref{eqn:proof symmetric set with no consecutive indices is schur positive - specific set}, we find that $|\A(D = \Des(T_\mu))| = c_\mu \neq 0$.

    We may conclude that there exists an element $a \in \A$ with $D(a) = \Des(T_\mu)$. The set $D(a)$ is sparse, so $\{1,2\} \nsubseteq D(a)$. Lemma~\ref{lem:superstandard characterize coefficients} implies that $\{1,2\} \subseteq \Des(T_\mu)$ whenever $\mu_1' >2$, so we may deduce that $\mu_1' \le 2$. Therefore, $c_\lambda = 0$ for every partition $\lambda$ with more than $2$ parts.
    
    Thus, we can reformulate Equation~\eqref{eqn:proof symmetric set with no consecutive indices is schur positive - specific set} as
    \begin{equation}\label{eqn:proof symmetric set with no consecutive indices is schur positive:equation with only two rows}
        |\A(D = J)| = \sum_{k=0}^n c_k \left|\Syt(N-k,k)(\Des = J)\right|.
    \end{equation}
    
    By Lemma~\ref{lem:superstandard two rows characterize coefficients}, since only partitions into at most two parts are involved in Equation~\eqref{eqn:proof symmetric set with no consecutive indices is schur positive:equation with only two rows}, substituting $J = \Des(T_{\lambda})$ for $\lambda = (N-k,k)$ yields $|\A(D = \odds(k))| = c_k$, as required.
\end{proof}

\begin{proof}[Second proof of Lemma~\ref{lem:if sparse and good sizes then schur positive two rows}]
    The lemma follows directly from Lemma~\ref{lem:sufficent condition symmetry two rows} together with Lemma~\ref{lem:if two rows and symmetric then schur positive}.
\end{proof}

\section{Short chords of matchings}
\label{sec:short chords in matchings}

In this section, we analyze the set of matchings $\M_{N,f}$ with respect to short chords (Recall Definition~\ref{def:matching} and Definition~\ref{def:short}). First, we apply Theorem~\ref{thm:criterion schur positive two rows} to establish Theorem~\ref{thm:matchings schur positive}, which asserts that $\M_{N,f}$ is Schur-positive. Next, we provide a bijective proof of Theorem~\ref{thm:matchings schur positive}, which will be utilized in Section~\ref{sec:refinements} to refine the Schur-positivity result. Finally, we demonstrate that the Schur expansion of $\Q_{\Short} (\M_{N,f})$ may be explicitly interpreted in terms of Bessel polynomials.

\subsection{First proof of Theorem~\ref{thm:matchings schur positive}}
\begin{proof}[\unskip\nopunct]
    Clearly, the function $\Short: \M_{N,f} \to 2^{[N-1]}$ is sparse (as defined in Definition~\ref{def:sparse}). Let $J = \{j_1, \dots, j_k\} \subseteq [N-1]$ be a sparse set, and let us enumerate the elements of $\M_{N,f}(\Short \supseteq J)$. In every matching in $\M_{N,f}(\Short \supseteq J)$, the vertices $j_i$ and $j_i + 1$ are matched for all $1 \le i \le k$, and the remaining $N-2k$ vertices can be matched in any way, subject to the condition that exactly $f$ vertices remain unmatched. Therefore, we have
    \[
        |\M_{N,f}(\Short \supseteq J)| = |\M_{N-2k,f}|,
    \]
    and thus it depends only on the size of $J$.

    By applying Theorem~\ref{thm:criterion schur positive two rows}, we conclude that $\M_{N,f}$ is Schur-positive with respect to $\Short$, with the following Schur expansion:
    \[
        \Q_{\Short} (\M_{N,f}) = \sum_{k=0}^n |\M_{N,f}(\Short = \{1,3,5,\dots, 2k-1\})|\ s_{N-k,k}.
    \]
    (While Theorem~\ref{thm:criterion schur positive two rows} sums over $0 \le k \le \lfloor\frac{N}{2} \rfloor$, this sum only goes up to $n = \frac{N-f}{2}$. However, the extra summands equal 0 and do not affect the expression.) Clearly,
    \[
        |\M_{N,f}(\Short = \{1,3,5,\dots, 2k-1\})| = |\M_{N-2k,f}(\Short = \emptyset)|,
    \]
    and we obtain the required Schur expansion.
\end{proof}

\subsection{Bijective proof of Theorem~\ref{thm:matchings schur positive}}
\label{subsec:bijective proof matchings schur positive}

Before presenting the bijective proof of Theorem~\ref{thm:matchings schur positive}, we give some definitions and notations for matchings:






\begin{definition}
\label{def:restriction of matching}
    Given a matching $m \in \M_N$, we say that a set $S \subseteq [N]$ is $m$-invariant if for every chord $(i,j) \in m$ we have $i \in S$ if and only if $j \in S$. Moreover, for a matching $m$ and an $m$-invariant set $S$ we define the \emph{restriction} of $m$ to $S$, denoted $\res_S(m)$, to be a matching on $S$ which is obtained by removing all the vertices not in $S$ from $m$.
\end{definition}

For example , consider the matching $m_1 = \{(1,2),\; (3,5),\; (4)\} \in \M_{5,1}$ as in Figure~\ref{fig:core first example before}. The set $S_1=\{1,2,4\}$ is $m_1$-invariant, and $\res_{S_1}(m_1)=\{(1,2),\; (4)\}$. On the other hand, the set $S_2 = \{1,2,3\}$ is not $m_1$-invariant, so $\res_{S_2}(m_1)$ is not defined.


\begin{observation}
\label{obs:restriction of matching to set and complement defines uniquely}
    Let $S \subseteq [N]$ be a set of vertices, let $m_1$ be a matching on $S$, and let $m_2$ be a matching on $[N]\setminus S$. Then there exists a unique matching $m \in \M_N$ such that $\res_S(m) = m_1$ and $\res_{[N] \setminus S}(m) = m_2$.
\end{observation}


Now let us present a bijection
\[
    F:\M_{N,f} \to \bigcup_{k=0}^n \M_{N-2k,f}(\Short = \emptyset) \times \Syt(N-k,k),
\]
that sends matchings $m \in \M_{N,f}$ to pairs $(m_0,T)$, where $m_0$ is a short-chord-free matching on $[N-2k]$ with $f$ unmatched vertices and $T\in \Syt(N-k,k)$ for some $0\le k \le n$, such that
$\Des T = \Short m$ for all $m$. By Theorem~\ref{thm:criterion schur positive young}, the existence of such a bijection implies Theorem~\ref{thm:matchings schur positive}.

\subsubsection{Constructing the bijection}
\label{subsec:constructing the bijection}
First, we define the core of a matching:
\begin{definition}
\label{def:core of matching}
    The reduction process for a given matching $m$ repeatedly removes any short chords of the matching until there are no short chords left. The remaining vertices are then re-indexed with natural numbers starting from $1$ while keeping their relative order, resulting in a matching denoted by $\core(m)$.
    
    A chord or vertex of $m$ is called \emph{stable} if it is not removed during the process, and unstable otherwise.
    The set of stable vertices of $m$ is denoted $\Stable(m)$.
\end{definition}

Proposition~\ref{prop:core well defined} below implies that $\core(m)$ and $\Stable(m)$ are well-defined and are independent of the order of the steps.

For a matching $m \in \M_{N,f}$, we also define $T(m)$ as the unique SYT consisting of $N$ cells arranged in two rows, such that
\[
    \row_2(T) = \{j \mid \text{the chord } (i,j) \in m \text{ is unstable}\}.
\]
Recall that when writing $(i,j) \in m$ we assume that $i<j$. Therefore, for every unstable chord $(i,j)\in m$ with $i<j$, we have $i \in \row_1(T(m))$ and $j \in \row_2(T(m))$. We define the bijection by $F(m) = (\core(m),T(m))$.

Next, we turn to provide examples of the bijection. Then, in the remaining of the subsection, we will prove that the bijection is well-defined and explore some of its properties. Section~\ref{subsec:proof of bijection} will be devoted to proving that $F$ is bijective by constructing its inverse function.

As a first example, consider the matching $m_1 = \{(1,2),\; (3,5),\; (4)\} \in \M_{5,1}$ as in Figure~\ref{fig:core first example before}.

\begin{figure}[htb]
    \begin{subfigure}{\textwidth}
        \begin{center}
            \begin{tikzpicture}[node distance={15mm}, thick, main/.style = {draw, circle, scale=1.2}]  
                \node[main] (1) {$1$};  
                \node[main] (2) [right of=1] {$2$}; 
                \node[main] (3) [right of=2] {$3$}; 
                \node[main] (4) [right of=3] {$4$}; 
                \node[main] (5) [right of=4] {$5$}; 
                
                \draw[densely dotted] (1) to [out=67,in=113,looseness=1] (2);  
                \draw (3) to [out=67,in=113,looseness=0.8] (5);
            \end{tikzpicture}
        \end{center}
        \caption{Plot of $m_1$}
        \label{fig:core first example before}
    \end{subfigure}
    \begin{subfigure}{\textwidth}
        \begin{center}
            \vspace{3ex}
            \begin{subfigure}{0.3\textwidth}
                \begin{tikzpicture}[node distance={15mm}, thick, main/.style = {draw, circle, scale=1.2}]  
                    \node[main] (1) {$1$};  
                    \node[main] (2) [right of=1] {$2$}; 
                    \node[main] (3) [right of=2] {$3$}; 
                    
                    \draw (1) to [out=67,in=113,looseness=0.8] (3);
                \end{tikzpicture}
            \end{subfigure}
            \begin{subfigure}{0.2\textwidth}
                \[
                    \young(1345,2)
                \]
                \vspace{0.1ex}
            \end{subfigure}
        \end{center}
        \caption{Plot of $\core(m_1)$ and $T(m_1)$}
        \label{fig:core first example core}
    \end{subfigure}
    \caption{Example of the bijection for $m_1 = \{(1,2),\; (3,5),\; (4)\} \in \M_{5,1}$}
    \label{fig:core first example}
\end{figure}

            
            
            
            

During the reduction process of $m_1$, the short chord $(1,2)$ is removed. Then, the vertices $3,4,5$ remain, so $\Stable(m_1) = \{3,4,5\}$. Next, the stable vertices are renumbered to $\{1,2,3\}$, so $\core(m_1) = \{(1,3),\; (2)\}$, as in Figure~\ref{fig:core first example core}. In addition, the only unstable chord of $m_1$ is $(1,2)$, so $\row_2(T(m_1))=\{2\}$. Therefore, $T(m_1)$ is the tableau presented in Figure~\ref{fig:core first example core}.

Next, consider the matching $m_2 = \{(1,7),\; (2,10),\; (3,6),\; (4,5),\; (8,9)\} \in \M_{10,0}$ as in Figure~\ref{fig:core second example before}.

\begin{figure}[htb]
    \begin{subfigure}{\textwidth}
        \begin{center}
            \begin{tikzpicture}[node distance={15mm}, thick, main/.style = {draw, circle, scale=1.1}]  
            \node[main] (1) {$1$};
            \node[main] (2) [right of=1] {$2$};
            \node[main] (3) [right of=2] {$3$};
            \node[main] (4) [right of=3] {$4$};
            \node[main] (5) [right of=4] {$5$};
            \node[main] (6) [right of=5] {$6$};
            \node[main] (7) [right of=6] {$7$};
            \node[main] (8) [right of=7] {$8$};
            \node[main] (9) [right of=8] {$9$};
            \node[main] (10) [right of=9] {$10$};
            
            \draw (1) to [out=50,in=130,looseness=0.6] (7);
            \draw (2) to [out=50,in=130,looseness=0.5] (10);
            \draw[densely dotted] (3) to [out=55,in=125,looseness=0.7] (6);
            \draw[densely dotted] (4) to [out=67,in=113,looseness=1] (5);
            \draw[densely dotted] (8) to [out=67,in=113,looseness=1] (9);
        \end{tikzpicture}
        \end{center}
        \caption{Plot of $m_2$}
        \label{fig:core second example before}
    \end{subfigure}
    \begin{subfigure}{\textwidth}
        \begin{center}
            \vspace{3ex}
            \begin{subfigure}{0.4\textwidth}
                \begin{tikzpicture}[node distance={15mm}, thick, main/.style = {draw, circle, scale=1.2}]  
                    \node[main] (1) {$1$};  
                    \node[main] (2) [right of=1] {$2$}; 
                    \node[main] (3) [right of=2] {$3$}; 
                    \node[main] (4) [right of=3] {$4$}; 
                    
                    \draw (1) to [out=67,in=113,looseness=0.8] (3);
                    \draw (2) to [out=67,in=113,looseness=0.8] (4);
                \end{tikzpicture}
            \end{subfigure}
            \begin{subfigure}{0.4\textwidth}
                \[
                    \newcommand\ten{10}
                    \young(123478\ten,569)
                \]
                \vspace{0.1ex}
            \end{subfigure}
        \end{center}
        \caption{Plot of $\core(m_2)$ and $T(m_2)$}
        \label{fig:core second example core}
    \end{subfigure}
    \caption{Example of the bijection for $m_2 = \{(1,7),\; (2,10),\; (3,6),\; (4,5),\; (8,9)\} \in \M_{10,0}$}
    \label{fig:core second example}
\end{figure}

During the reduction process of $m_2$, we first remove the short chords $(4,5)$ and $(8,9)$. Then, the chord $(3,6)$ becomes short and is subsequently removed as well. The remaining vertices are $1,2,7$ and $10$, so $\Stable(m_2) = \{1,2,7,10\}$. Next, the stable vertices are renumbered to $[4]$, so $\core(m_2) = \{(1,3),\; (2,4)\}$, as in Figure~\ref{fig:core second example core}. In addition, the unstable chords of $m_2$ are $(3,6),\; (4,5)$ and $(8,9)$, so $\row_2(T(m_2))=\{5,6,9\}$. Therefore, $T(m_2)$ is the tableau presented in Figure~\ref{fig:core second example core}.
            
            

Moving on to proving that the bijection $F$ is well-defined, first let us prove that the core of a matching is well-defined:

\begin{proposition}
\label{prop:core well defined}
    Given a matching, its stable vertices and its core are well-defined.
\end{proposition}
\begin{proof}
    Let $m \in \M_N$ be a matching. Denote the chords that are removed during a given reduction process of $m$ by $e_1, \dots, e_k$, and denote $e_\ell = (i_\ell,j_\ell)$. Moreover, denote the chords that are removed during another reduction process by $e_1', \dots, e_{k'}'$, and denote $e_\ell' = (i_\ell',j_\ell')$. It suffices to prove that $e_\ell \in \{e_1', \dots, e_{k'}'\}$ for all $1 \le \ell \le k$ (i.e., every chord that is removed during the first reduction process is removed during the second process as well). Assume by contradiction that $e_\ell \notin \{e_1', \dots, e_{k'}'\}$ for some $\ell$, and denote by $\ell_0$ the minimal such $\ell$.
    That is, $e_1,\dots,e_{\ell_0-1} \in \{e_1', \dots, e_{k'}'\}$. Removing the chords $e_1, \dots, e_k$ from $m$ constitutes a valid reduction process, so the chord $e_{\ell_0}$ becomes short before it is removed. That is, $i \in \{i_1,j_1,\dots,i_{\ell_0-1},j_{\ell_0-1}\}$ for all $i_{\ell_0} < i < j_{\ell_0}$. Therefore, the chord $e_{\ell_0}$ is a short chord of the core obtained by removing $e_1', \dots, e_{k'}'$ of $m$, contradicting the requirement that the reduction process continues until there are no short chords remaining. Therefore, we may conclude that if a chord is removed during a reduction process then it is removed during any reduction process, and $\core(m)$ and $\Stable(m)$ are well-defined.
\end{proof}
\begin{corollary}
    Let $N$ and $f$ be nonnegative integers. Then the function
    \[
        F:\M_{N,f} \to \bigcup_{k=0}^n \M_{N-2k,f}(\Short = \emptyset) \times \Syt(N-k,k)
    \]
    defined by $F(m) = (\core(m),T(m))$ is well-defined.
\end{corollary}
\begin{proof}
    Let $m\in \M_{N,f}$ be a matching. By Proposition~\ref{prop:core well defined}, we obtain that $\core(m)$ and $\Stable(m)$ are well-defined. Therefore, the two-row tableau $T:=T(m)$ consisting of $N$ cells that is defined by
    \[
        \row_2(T) = \{j \mid \text{the chord } (i,j) \in m \text{ is unstable}\}
    \]
    is also well-defined. In order to prove that $F$ is well-defined, it remains to show that $T(m) \in \Syt(N-k,k)$, where $k$ is the number of unstable chords of $m$.

    Obviously, every letter $i \in [N]$ appears exactly once in $T(m)$, and the rows are increasing. Notice that for every entry $j \in \row_2(T(m))$ there exists $i<j$ such that $(i,j) \in m$ is an unstable chord. Therefore, every entry $j \in \row_2(T(m))$ is associated to an entry $i \in \row_1(T(m))$ such that $i<j$. Moreover, if $j\ne j' \in \row_2(T(m))$ and $(i,j),\; (i',j') \in m$ then $i \ne i'$, so the columns of $T(m)$ are increasing and $T(m)$ is a standard Young tableau. Finally, $|\row_2(T(m))| = k$, implying that $T(m) \in \Syt(N-k,k)$.
\end{proof}

After establishing that $F$ is a valid function, we turn our attention to exploring some of its properties.

\begin{lemma}
\label{lem:properties of core}
    The reduction process has the following properties:
    \begin{enumerate}[label=(\arabic*), ref=\ref{lem:properties of core}\text{ (part~}\arabic*)]
        \item \label{lem:properties of core--intersect stable} If a chord intersects another chord then it is stable.
        \item \label{lem:properties of core--chord unstable iff all midpoints unstable} A chord $(i,j)$ is stable if and only if there exists a stable vertex in $\{i+1,\dots,j-1\}$.
        \item \label{lem:properties of core--perfect non-crossing unstable} Given a matching $m$ and $i < j$, if the set $[i,j]$ is $m$-invariant \textup{(}as defined in Definition~\ref{def:restriction of matching}\textup{)} and the restricted matching $\res_{[i,j]}(m)$ is perfect and non-crossing, then $\ell \notin \Stable(m)$ for all $\ell \in [i,j]$.
    \end{enumerate}
\end{lemma}

\begin{proof} \hfill 
    \begin{enumerate}
        \item Assume that $(i_1,i_3),\; (i_2,i_4) \in m$ for $i_1 < i_2 < i_3 < i_4$. As long as the chord $(i_2,i_4)$ is not removed, the chord $(i_1,i_3)$ does not become short and cannot be removed. On the other hand, as long as the chord $(i_1,i_3)$ is not removed, the chord $(i_2,i_4)$ cannot be removed. Therefore, both chords cannot be removed during the reduction process.
        \item If the chord $(i,j)$ is unstable, then after removing some unstable vertices it becomes short, implying that every vertex between $i$ and $j$ is unstable. On the other hand, 
        if all the vertices between $i$ and $j$ are unstable, then they will eventually be removed, making the chord $(i,j)$ short, so the chord $(i,j)$ is unstable too.
        \item 
        Since the set $[i,j]$ is $m$-invariant and the restricted matching is perfect, Definition~\ref{def:restriction of matching} implies that for every $i \le i_1 \le j$ there exists $i \le i_2 \le j$ such that $i_1 \ne i_2$ and $(i_1,i_2) \in m$ or $(i_2,i_1) \in m$.
        Notice that the restricted matching $\res_S(m)$ where $S = \Stable(m) \cap [i,j]$ (i.e., the matching that consists of the stable chords $(i_1,i_2) \in m$ with $i \le i_1<i_2 \le j$) is short-chord-free, non-crossing and perfect. The only such a matching is the empty matching $\emptyset \in \M_0$, so every $i \le i_1 \le j$ is unstable.
        \qedhere
    \end{enumerate}
\end{proof}

As we proceed to apply Theorem~\ref{thm:criterion schur positive young} and establish the Schur-positivity of $\M_{N,f}$ with respect to short chords, let us prove that $F$ sends short chords of matchings to descents of SYTs, in the following sense:

\begin{proposition}
\label{prop:F statistic preserving}
    Let $m \in \M_{N,f}$ be a matching, and denote $T := T(m)$. Then $\Des T = \Short m$.
\end{proposition}
\begin{proof}
    Let $i \in \Short m$ be an index representing a short chord $(i,i+1) \in m$. The chord $(i,i+1)$ is unstable, so we may deduce from the definition of $T(m)$ that $i \in \row_1(T)$ and $i+1 \in \row_2(T)$. Therefore, $i \in \Des T$.
    
    On the other hand, let $i \in \Des T$ be a descent of $T$. This implies that $i \in \row_1(T)$ and $i+1 \in \row_2(T)$. Since $i+1 \in \row_2(T)$, we obtain that there exists $j < i+1$ such that $(j,i+1) \in m$ is an unstable chord. Since the chord $(j,i+1)$ is unstable and $j \le i \le i+1$, we obtain by Lemma~\ref{lem:properties of core--chord unstable iff all midpoints unstable} that $i$ is an endpoint of an unstable chord of $m$. Since $i \in \row_1(T)$, we obtain that $i$ opens an unstable chord $(i,j')$ of $m$ for some $j'>i$.
    If $(i,i+1) \notin m$ then $j<i$ and $j' > i+1$, and we obtain that the unstable chord $(j,i+1)$ intersects the chord $(i,j')$, contradicting Lemma~\ref{lem:properties of core--intersect stable}. Therefore, we may conclude that $(i,i+1) \in m$ and $i \in \Short m$.
\end{proof}




\subsubsection{Proof of bijection}
\label{subsec:proof of bijection}

In this section we will prove that the transformation $F$ defined in Section~\ref{subsec:constructing the bijection} is indeed a bijection, by constructing its inverse function.

In order to construct the inverse function, we will establish a correspondence between standard Young tableaux of two rows and ballot paths. We define ballot paths as follows:
\begin{definition}
\label{def:ballot path}
    Let $N \in \NN$ be a nonnegative integer. A \emph{ballot path} of length $N$ is a sequence of $N$ steps, where each step is either $(1,1)$ or $(1,-1)$. The path starts at the origin $(0,0)$. Namely, each step either moves one unit up and one unit right, or one unit down and one unit right. The path is said to be valid if it never goes below the x-axis, i.e., the y-coordinate of a point on the path is always non-negative.

    The set of ballot paths from $(0,0)$ to $(N,t)$ is denoted $\Pcal_{N,t}$. Given a ballot path $p \in \Pcal_{N,t}$, denote by $p_i$ the y-coordinate of $p$ after $i$ steps; in particular, $p_0=0$. The set $\UP(p) \subseteq [N]$ ($\DOWN(p) \subseteq [N]$) consists of the indices $i$ such that $p_i > p_{i-1}$ (respectively, $p_i < p_{i-1}$). Finally, define the height of the $i$-th \emph{step} of a path $p$ to be the maximum height of its two endpoints, and denote it by $\height_p(i) := \max(p_{i-1},p_i)$.
\end{definition}

The bijection between SYTs of two rows and ballot paths is direct: Associate $T \in \Syt(N-k,k)$ with the path $p(T) \in \Pcal_{N,N-2k}$ such that $\UP(p(T)) = \row_1(T)$ and $\DOWN(p(T)) = \row_2(T)$.

For example, consider the tableau $T \in \Syt(7,3)$ described in Figure~\ref{fig:core second example core}, with $\row_2(T) = \{5,6,9\}$. It is associated with the ballot path $p := p(T) \in \Pcal_{10,4}$ presented in Figure~\ref{fig:ballot path example}. The ballot path $p$ has, for example, $2 \in \UP(p)$ because $p_2 = 2 > p_1 = 1$. Conversely, $5 \in \DOWN(p)$ because $p_5 = 3 < p_4 = 4$. In addition, $\height_p(3) = \max(p_2,p_3) = 3$ while $\height_p(5) = \max(p_4,p_5) = 4$.

\begin{figure}
    \centering
    \[\begin{tikzcd}[column sep=small]
	{} &&&& \bullet &&&& \bullet && \bullet \\
	&&& \bullet && \bullet && \bullet && \bullet & {} \\
	&& \bullet &&&& \bullet &&&& {} \\
	& \bullet &&&&&&&&& {} \\
	{} &&&&&&&&&&& {}
	\arrow[shorten >= -5pt, shorten <= -5pt, "2", no head, from=4-2, to=3-3, line width=2.5pt]
	\arrow[shorten >= -5pt, shorten <= -5pt, ""{name=0, anchor=center, inner sep=0}, "3", no head, from=3-3, to=2-4]
	\arrow[shorten >= -5pt, shorten <= -5pt, ""{name=1, anchor=center, inner sep=0}, "4", no head, from=2-4, to=1-5]
	\arrow[shorten >= -5pt, shorten <= -5pt, ""{name=2, anchor=center, inner sep=0}, "6", no head, from=2-6, to=3-7]
	\arrow[shorten >= -5pt, shorten <= -5pt, "7", no head, from=3-7, to=2-8, line width=2.5pt]
	\arrow[shorten >= -5pt, shorten <= -5pt, ""{name=3, anchor=center, inner sep=0}, "8", no head, from=2-8, to=1-9]
	\arrow[shorten >= -5pt, shorten <= -5pt, ""{name=4, anchor=center, inner sep=0}, "9", no head, from=1-9, to=2-10]
	\arrow[shorten >= -5pt, shorten <= -5pt, "10", no head, from=2-10, to=1-11, line width=2.5pt]
	\arrow[shorten >= -5pt, shorten <= -5pt, from=5-1, to=5-12]
	\arrow[shorten >= -5pt, shorten <= -2pt, from=5-1, to=1-1]
	\arrow[shorten >= -5pt, shorten <= -5pt, ""{name=5, anchor=center, inner sep=0}, "5", no head, from=1-5, to=2-6]
	\arrow[shorten >= -5pt, shorten <= -2pt, "1", no head, from=5-1, to=4-2, line width=2.5pt]
	\arrow[shorten >= -5pt, shorten <= -5pt, loosely dashed, no head, from=4-2, to=4-11]
	\arrow[shorten >= -5pt, shorten <= -5pt, loosely dashed, no head, from=3-3, to=3-11]
	\arrow[shorten >= -5pt, shorten <= -5pt, loosely dashed, no head, from=2-8, to=2-11]
	\arrow[dotted, no head, from=0, to=2]
	\arrow[dotted, no head, from=3, to=4]
	\arrow[dotted, no head, from=1, to=5]
    \end{tikzcd}\]
    \caption{A ballot path $p  \in \Pcal_{10,4}$}
    \label{fig:ballot path example}
\end{figure}

Next, for a given ballot path $p \in \Pcal_{N,t}$, we construct a set $\Stable(p)$ and a perfect matching $m_{\text{unstable}}(p)$ on $[N]\setminus\Stable(p)$ as follows: For every vertex $j \in \DOWN(p)$, we match it in $m_{\text{unstable}}(p)$ to the maximal $i < j$ such that $\height_p(i)=\height_p(j)$. Since $p$ is a valid ballot path, it can be deduced that for every $j \in \DOWN(p)$, there exists a unique $i<j$ satisfying $\height_p(i)=\height_p(j)$ such that $i$ is maximal among all such elements in $[N]$. Additionally, it can be inferred from the discrete continuity of the path that this $i$ is necessarily a part of $\UP(p)$. The set $\Stable(p)$ consists of all $i \in \UP(p)$ that are not involved in any chord in $m_{\text{unstable}}(p)$. Notice that $i \in \Stable(p)$ if and only if $i \in \UP(p)$ and $\height_p(j) > \height_p(i)$ for all $j>i$, so $\left|\Stable(p)\right| = t$.


We are now ready to describe the inverse bijection of $F$, denoted
\[
    \tilde{F}:\bigcup_{k=0}^n \M_{N-2k,f}(\Short = \emptyset) \times \Syt(N-k,k) \to \M_{N,f}.
\]
Given a short-chord-free matching $m_0 \in \M_{N-2k,f}(\Short = \emptyset)$ and a tableau $T \in \Syt(N-k,k)$ for some $k$, denote $p = p(T) \in \Pcal_{N,N-2k}$.
We will construct a matching $m_{\text{stable}}$ on $\Stable(p)$ and a matching $m_{\text{unstable}}$ on $[N]\setminus\Stable(p)$, and then apply Observation~\ref{obs:restriction of matching to set and complement defines uniquely} to obtain $\tilde{F}(m_0,T) \in \M_{N,f}$. We construct these sub-matchings as follows:
\begin{itemize}
    \item $m_{\text{stable}}$: Since $p \in \Pcal_{N,N-2k}$, we infer that $\left|\Stable(p)\right| = N-2k$. Therefore, we may rename the vertices of $m_0 \in \M_{N-2k,f}$ to $\Stable(p)$ as follows: There exists a unique bijection $\varphi:[N-2k] \to \Stable(p)$ such that $i < j$ if and only if $\varphi(i) <\varphi(j)$ for all $i,j$. The matching $m_{\text{stable}}$ on $\Stable(p)$ consists of the chords $(\varphi(i), \varphi(j))$ for all $(i,j) \in m_0$.
    \item $m_{\text{unstable}} = m_{\text{unstable}}(p)$ is the matching described earlier.
\end{itemize}

For example, consider $m_0 = \{(1,3),\; (2,4)\} \in \M_{4,0}(\Short = \emptyset)$ and $T \in \Syt(N-k,k)$ presented in Figure~\ref{fig:core second example core}. As mentioned before, $T$ is associated with the ballot path $p := p(T) \in \Pcal_{10,4}$ presented in Figure~\ref{fig:ballot path example}. Therefore, we obtain the matching $m_{\text{unstable}}(p) = \{(3,6),\; (4,5),\; (8,9)\}$ (with the pairs of steps that correspond to its chords connected by dotted lines in the figure) and $\Stable(p) = \{1,2,7,10\}$ (with the steps that correspond to these vertices denoted by bold lines). Applying the order-preserving bijection $\varphi:[4] \to \{1,2,7,10\}$ on $m_0$ yields the matching $m_{\text{stable}} = \{(1,7),\; (2,10)\}$. Therefore,
\[
    \tilde{F}(m_0,T) = \{(1,7),\; (2,10),\; (3,6),\; (4,5),\; (8,9)\}
\]
is the matching presented in Figure~\ref{fig:core second example before}.

It remains to show that $\tilde{F}$ is indeed the inverse function of $F$. We will do so in two steps.

\proofstep{Step 1: $\tilde{F}\circ F = Id$}
\begin{lemma}
\label{lem:F_tilde(F)=id}
    Let $m \in \M_{N,f}$ be a matching, and denote $F(m) = (\core(m),T)$. Then $\tilde{F}(\core(m),T) = m$.
\end{lemma}
\begin{proof}
    Denote $\left|\Stable(m)\right| = N-2k$, implying that $\core(m) \in \M_{N-2k,f}$ and $T \in \Syt(N-k,k)$. Moreover, denote $p=p(T) \in \Pcal_{N,N-2k}$. First, we prove that
    \begin{equation}
    \label{eqn:lemma F_tilde(F)=id}
        \res_{[N]\setminus\Stable(m)}(m) = m_{\text{unstable}}(p).
    \end{equation}
    We note that the set $[N] \setminus \Stable(m)$ is $m$-invariant, so the left-hand side of Equation~\eqref{eqn:lemma F_tilde(F)=id} is well-defined. Notice that $|[N]\setminus\Stable(m)| = |[N]\setminus\Stable(p)| = 2k$ and that $\res_{[N]\setminus\Stable(m)}(m)$ is a perfect matching. Therefore, in order to prove that Equation~\eqref{eqn:lemma F_tilde(F)=id} holds, it suffices to prove that any unstable chord of $m$ belongs to $m_{\text{unstable}}(p)$ too. Let $(i,j) \in m$ be an unstable chord. Thus, we may deduce that $i \in \UP(p)$ and $j \in \DOWN(p)$. Let $i' \in \UP(p)$ such that $i < i' < j$. By Lemma~\ref{lem:properties of core--chord unstable iff all midpoints unstable}, $i'$ is an endpoint of an unstable chord $(i',j') \in m$ with $i' < j'$. By Lemma~\ref{lem:properties of core--intersect stable}, we obtain that the chord $(i,j)$ does not intersect $(i',j')$, implying that $i < i' < j' < j$. On the other hand, every $j' \in \DOWN(p)$ with $i < j' < j$ is an endpoint of an unstable chord $(i',j') \in m$ with $i < i' < j' < j$. Therefore, we obtain a bijection from $\UP(p) \cap [i,j]$ to $\DOWN(p) \cap [i,j]$ (with $i$ matched with $j$), so $\height_p(i) = \height_p(j)$. Moreover, this bijection has the property that if $i' \in \UP(p) \cap [i,j]$ is matched with $j' \in \DOWN(p) \cap [i,j]$ then $i' < j'$ (i.e., the ascending steps of the path appear before the associated descending steps), so $\height_p(j) < \height_p(i')$ for all $i' \in \UP(p) \cap [i+1,j-1]$. We may conclude that $(i,j) \in m_{\text{unstable}}(p)$ for every unstable chord $(i,j)$ of $m$ and prove that Equation~\eqref{eqn:lemma F_tilde(F)=id} holds.

    Next, we denote $m' = \tilde{F}(\core(m),T)$ and prove that $m=m'$. From Equation~\eqref{eqn:lemma F_tilde(F)=id} we deduce that the supports of the matchings $\res_{[N]\setminus\Stable(m)}(m)$ and $m_{\text{unstable}}(p)$ are identical, and therefore $\Stable(m) = \Stable(p)$. Thus, the set $[N]\setminus\Stable(m)$ is both $m$-invariant and $m'$-invariant, and restricting each of these matchings to $[N]\setminus\Stable(m)$ results in $m_{\text{unstable}}(p)$. In addition, it can be easily verified from the descriptions of $F$ and $\tilde{F}$ that $\res_{\Stable(m)}(m) = \res_{\Stable(m)}(m')$ is the matching obtained by relabeling the vertices of $\core(m)$ with the elements of $\Stable(m)$ in increasing order. By Observation~\ref{obs:restriction of matching to set and complement defines uniquely}, we may deduce that $m=m'$.
\end{proof}

\proofstep{Step 2: $F \circ \tilde{F} = Id$}
\begin{lemma}
\label{lem:F(F_tilde)=id}
    Let $m_0 \in \M_{N-2k,f}$ with $\Short(m_0) = \emptyset$ and $T \in \Syt(N-k,k)$ for some $k$, and denote $m = \tilde{F}(m_0,T)$. Then $\core(m)=m_0$ and $T(m) = T$.
\end{lemma}
\begin{proof}
    Denote $p=p(T) \in \Pcal_{N,N-2k}$ and $\Stable(p) = \{i_1, \dots, i_{N-2k}\}$ where $i_1 < \cdots i_{N-2k}$. We first prove that $\Stable(m) = \Stable(p)$. Notice that the matching $m_{\text{unstable}}(p)$ is non-crossing. This can be viewed visually from Figure~\ref{fig:ballot path example}, where $m_{\text{unstable}}(p)$ is denoted by horizontal dotted lines that cross the path only in their endpoints. Indeed, let $j_1 < j_2 < j_3 < j_4$ be four vertices, and assume, by contradiction, that both $(j_1,j_3)$ and $(j_2,j_4)$ belong to $m_{\text{unstable}}(p)$. By the definition of $m_{\text{unstable}}(p)$, the assumption $(j_1,j_3) \in m_{\text{unstable}}(p)$ implies that $\height_p(j_1) = \height_p(j_3)$, and $\height_p(j) \ne \height_p(j_1)$ for all $j_1 < j < j_3$. Since $\height_p(j_1+1) \ge \height_p(j_1)$ and due to the discrete continuity of the path, we obtain that $\height_p(j) > \height_p(j_1)$ for all $j_1 < j < j_3$. Consequently, we obtain $\height_p(j_2) > \height_p(j_1) = \height_p(j_3)$. Similarly, the assumption $(j_2,j_4) \in m_{\text{unstable}}(p)$ implies that $\height_p(j_3) > \height_p(j_2)$, in contradiction. Therefore, we may conclude that the matching $m_{\text{unstable}}(p)$ is non-crossing.
    
    Next, we may infer that for every $1 \le \ell < N-2k$, the segment $[i_\ell+1,i_{\ell+1}-1]$ is $m$-invariant and the restricted matching $\res_{[i_\ell+1,i_{\ell+1}-1]}(m)$ is perfect and non-crossing. By Lemma~\ref{lem:properties of core--perfect non-crossing unstable}, we obtain that $j \notin \Stable(m)$ for all $i_\ell < j < i_{\ell+1}$. Similarly, we obtain that if $j < i_1$ or $j > i_{N-2k}$ then $j \notin \Stable(m)$, and therefore $\Stable(m) \subseteq \Stable(p)$. Thus, a valid reduction process of $m$ may begin with removing every vertex not in $\Stable(p)$. We may deduce from the description of $\tilde{F}$ that $\res_{\Stable(p)}(m)$ is the matching obtained by relabeling the vertices of $m_0$ with the elements of $\Stable(p)$ in increasing order. This matching is short-chord-free, so $\Stable(m) = \Stable(p)$ and $\core(m) = m_0$.

    It remains to prove that $T(m) = T$, namely that $p' = p$, where $p' := p(T(m))$. Since $\Stable(m) = \Stable(p)$, we may infer from the description of $\tilde{F}$ that
    \[
        \DOWN(p) = \{j \mid \text{the chord } (i,j) \in m \text{ is unstable}\}.
    \]
    Thus, $\DOWN(p) = \DOWN(p')$ and therefore $p = p'$.
\end{proof}

Finally, we conclude the bijective proof:
\begin{proof}[Bijective proof of Theorem~\ref{thm:matchings schur positive}]
    By Theorem~\ref{thm:criterion schur positive young}, it suffices to prove that
    \[
        \sum_{m \in \M_{N,f}} \boldsymbol{t}^{\Short(m)} = \sum_{k=0}^N |\M_{N-2k,f}(\Short = \emptyset)| \sum_{T \in \Syt(N-k,k)} \boldsymbol{t}^{\Des(T)},
    \]
    where $\boldsymbol{t}^J := \prod_{j \in J} t_j$ for $J \subseteq [N-1]$.
    Equivalently, it suffices to present a bijection between $\M_{N,f}$ and $\bigcup_{k=0}^n \M_{N-2k,f}(\Short = \emptyset) \times \Syt(N-k,k)$, such that if $m \mapsto (m_0,T)$ then $\Short(m) = \Des(T)$. The transformation $F$ defined by $F(m) = (\core(m),T(m))$ is a bijection by Lemma~\ref{lem:F_tilde(F)=id} combined with Lemma~\ref{lem:F(F_tilde)=id}, and it satisfies $\Short(m) = \Des(T(m))$ by Proposition~\ref{prop:F statistic preserving}.
\end{proof}

\subsection{Analysis of the coefficients and relations with Bessel polynomials}
\label{subsec:schur coefficients bessel}
The \emph{Bessel polynomials} $\theta_n(x)$, sometimes called the \emph{reverse Bessel polynomials}, are given by the generating function
\[
    \frac{1}{\sqrt{1-2v}} \exp\left[x(1-\sqrt{1-2v})\right] = \sum_{n=0}^\infty \frac{v^n}{n!} \theta_n(x).
\]
They also have the explicit formula
\[
    \theta_n(x) = \sum_{k=0}^n \frac{(2n-k)!}{k!(n-k)!2^{n-k}} x^k.
\]
For more information about the Bessel polynomials, the reader is referred to~\cite{bessel_polynomials_book}.

McSorley and Feinsilver \cite[Theorem 3.5]{bessel_polynomials_mathcings} discovered the following identity:
\begin{theorem}[McSorley and Feinsilver]
    For every $n \ge 0$:
    \[
        \theta_n(x-1) = \sum_{i=0}^n h(P_{2n},i)x^i,
    \]
    where $h(P_{2n},i)$ is the number of perfect matchings on $2n$ vertices with $i$ short chords.
\end{theorem}

An equivalent formulation of this result states that
\[
    \theta_n(x) = \sum_{i=0}^n h(P_{2n},i)(x+1)^i,
\]
so the sequence $h(P_{2n},i)$ can be thought of as the coefficients of the Taylor expansion of $\theta_n(x)$ around $x=-1$.

\begin{observation}
    For every $n \ge 0$:
    \[
        h(P_{2n},i) = |\M_{2n-i,i} (\Short=\emptyset)|.
    \]
\end{observation}
\begin{proof}
    We give a bijective proof for the statement. The bijection sends a perfect matching $m \in \M_{2n,0}$ with $i$ short chords to a short-chord-free matching on $2n-i$ vertices with $i$ unmatched vertices, by replacing every short chord with an unmatched vertex. For example, the perfect matching $\{(1,5),\; (2,3),\; (4,6)\} \in \M_{6,0}$ is sent to $\{(1,4),\; (2),\; (3,5)\} \in \M_{5,1}$. Clearly, this is a bijection between the two desired sets.
\end{proof}

Therefore, we can reformulate Theorem~\ref{thm:matchings schur positive} and obtain:
\begin{corollary}
\label{cor:schur expansion bessel}
    Let $n,f \in \NN$ be nonnegative integers, and denote $N=2n+f$. Then the Schur expansion of the set $\M_{N,f}$ with respect to $\Short$ is given by the formula
    \[
        \Q_{\Short} (\M_{N,f}) = \sum_{k=0}^n h(P_{N+f-2k},f) s_{N-k,k},
    \]
    where $h(P_{N+f-2k},f)$ is the coefficient of $(x+1)^f$ in the Taylor expansion of the Bessel polynomial $\theta_{n+f-k}(x)$ around $x=-1$
\end{corollary}

\section{Refinements of the bijection}
\label{sec:refinements}
In this section, we will utilize the bijection $F$ discussed in Section~\ref{subsec:bijective proof matchings schur positive} to refine Theorem~\ref{thm:matchings schur positive} and find many Schur-positive sets of matchings with respect to the set of short chords.
Indeed, given non-negative integers $N$ and $k$, for every short-chord-free matching $m_0 \in \M_{N-2k}(\Short = \emptyset)$, the set
\[
    \{m \in \M_{N} \mid \core(m) = m_0 \}
\]
is Schur-positive (as we will see later in Corollary~\ref{cor:closed under knuth like is schur positive}).

\subsection{Sets closed under Knuth equivalence}
As a first refinement of the Schur-positivity of $\M_{N,f}$, we study the Knuth equivalence of matchings presented in Definition \ref{def:knuth like equivalence}.
The power of this notion is reflected by the following theorem:
\begin{theorem}
\label{thm:knuth like equivalent iff same core}
    Two matchings $m_1, m_2 \in \M_N$ are Knuth equivalent if and only if $\core(m_1) = \core(m_2)$.
\end{theorem}

We will prove Theorem~\ref{thm:knuth like equivalent iff same core} in two steps:
\proofstep{Step 1: If two matchings are equivalent then they have the same core}
\begin{lemma}
    Let $m_1, m_2 \in \M_N$ be two Knuth equivalent matchings. Then $\core(m_1) = \core(m_2)$.
\end{lemma}
\begin{proof}
    Since $m_1$ and $m_2$ are equivalent, we deduce that $m_2$ can be obtained from $m_1$ by a sequence of elementary Knuth transformations. Notably, given a matching $m \in \M_N$ and a matching $\varphi(m)$ obtained from $m$ by applying an elementary Knuth transformation, the matchings $m$ and $\varphi(m)$ differ only in the relative position of a certain short chord, and therefore $\core(\varphi(m)) = \core(m)$. A direct induction shows that $\core(m_1) = \core(m_2)$.
\end{proof}

\proofstep{Step 2: If two matchings have the same core then they are equivalent}
\begin{lemma}
\label{lem:if same core then equivalent}
    Let $m_1, m_2 \in \M_N$ be matchings, and assume that $\core(m_1) = \core(m_2)$. Then $m_1$ and $m_2$ are Knuth equivalent.
\end{lemma}
Before proving Lemma~\ref{lem:if same core then equivalent}, we introduce the notion of inserting a short chord into a matching:
\begin{definition}
\label{def:insert short chord to matching}
    Let $m \in \M_N$ be a matching, and let $1 \le i \le N+1$ be an index. Denote by $\insrt_i(m) \in \M_{N+2}$ the matching obtained by \emph{inserting a short chord} that matches the vertices $i$ and $i+1$, while pushing every vertex $j \ge i$ to position $j+2$. Formally, denote by $f_i:[N] \to [N+2]\setminus\{i,i+1\}$ the function that is described as follows:
    \[
        f_i(j) = \begin{cases}
			j & \text{if $j < i$,}\\
            j+2 & \text{if $j \ge i$.}
		 \end{cases}
    \]
    Then the matching $\insrt_i(m)$ consists of the chords $(f_i(j_1),f_i(j_2))$ for all $(j_1,j_2) \in m$ together with $(i,i+1)$, and consists of the unmatched vertices $(f_i(j))$ for all $(j) \in m$.
\end{definition}

For example, if $m = \{(1,3),\; (2,6),\; (4,5)\}$, then $\insrt_3(m) = \{(1,5),\; (2,8),\; (3,4),\; (6,7)\}$.

The insertion function is closely related to Knuth equivalence, as demonstrated by the following lemmas:

\begin{lemma}
\label{lem:insert i equivalent insert j}
    Let $m \in \M_N$ be a matching and let $1 \le i,j \le N+1$ be indices. Then $\insrt_i(m)$ is Knuth equivalent to $\insrt_j(m)$.
\end{lemma}
\begin{proof}
    We may assume without loss of generality that $i \le j$, and prove the statement by induction on $j-i$. The statement is obvious for $i=j$.

    Assume that $i < j$. By Definition~\ref{def:insert short chord to matching}, the matching $\insrt_{i+1}(m)$ is Knuth equivalent to $\insrt_i(m)$. By the induction hypothesis, $\insrt_{i+1}(m)$ is Knuth equivalent to $\insrt_j(m)$ as well. Therefore, $\insrt_i(m)$ is Knuth equivalent to $\insrt_j(m)$.
\end{proof}

\begin{lemma}
\label{lem:insert same chord to equivalent remains equivalent}
    Let $m_1,m_2 \in \M_N$ be Knuth equivalent matchings, and let $1 \le i \le N+1$ be an index. Then $\insrt_i(m_1)$ and $\insrt_i(m_2)$ are Knuth equivalent.
\end{lemma}

\begin{proof}
    By Lemma~\ref{lem:insert i equivalent insert j}, we may assume that $i = N+1$. Assume that $m_2 = \varphi_1 \cdots \varphi_\ell (m_1)$ for some elementary Knuth transformations $\varphi_1, \dots, \varphi_\ell$. We prove the statement by induction on $\ell$. The statement is obvious for $\ell = 0$.

    Assume that $\ell > 0$. By the induction hypothesis, we may assume that $\insrt_{N+1}(m_2)$ is equivalent to $\insrt_{N+1}(\varphi_\ell(m_1))$. Therefore, it suffices to show that $\insrt_{N+1}(m_1)$ is equivalent to $\insrt_{N+1}(\varphi_\ell(m_1))$. Assume that $(j,j+1),\; (j+2,j') \in m_1$ for some $1 \le j,j' \le N$, and that $(j,j'),\; (j+1,j+2) \in \varphi_\ell(m_1)$, as other types of elementary Knuth transformations are handled similarly. Notice that $\insrt_{N+1}(m_1)$ and $\insrt_{N+1}(\varphi_\ell(m_1))$ have all but four chords in common. Specifically, $\insrt_{N+1}(m_1)$ contains the chords $(j,j+1)$ and $(j+2,j')$, while $\insrt_{N+1}(\varphi_\ell(m_1))$ contains the chords $(j,j')$ and $(j+1,j+2)$. Therefore, there exists an elementary Knuth transformation $\varphi$, such that $\varphi(\insrt_{N+1}(m_1)) = \insrt_{N+1}(\varphi_\ell(m_1))$. Thus, the matchings $\insrt_{N+1}(m_1)$ and $\insrt_{N+1}(\varphi_\ell(m_1))$ are equivalent, as required.
\end{proof}

We can combine Lemma~\ref{lem:insert i equivalent insert j} with Lemma~\ref{lem:insert same chord to equivalent remains equivalent} to obtain the following:
\begin{lemma}
\label{lem:insert many chords to matching is equivalent}
    Let $m \in \M_n$ be a matching, and let $i_1,\dots, i_k$ and $j_1,\dots, j_k$ be two sequences. Then the matchings $\insrt_{i_1}\cdots\insrt_{i_k}(m)$ and $\insrt_{j_1}\cdots\insrt_{j_k}(m)$ are Knuth equivalent.
\end{lemma}

\begin{proof}
    We prove the statement by induction on $k$. If $k=0$ then the statement is obvious.

    Assume that $k > 0$, and denote $m_i = \insrt_{i_2}\cdots\insrt_{i_k}(m)$ and $m_j = \insrt_{j_2}\cdots\insrt_{j_k}(m)$. We aim to prove that the matchings $\insrt_{i_1}(m_i)$ and $\insrt_{j_1}(m_j)$ are Knuth equivalent. By the induction hypothesis, $m_i$ and $m_j$ are equivalent. Therefore, by Lemma~\ref{lem:insert same chord to equivalent remains equivalent}, the matchings $\insrt_{i_1}(m_i)$ and $\insrt_{i_1}(m_j)$ are equivalent too. In addition, by Lemma~\ref{lem:insert i equivalent insert j}, the matchings $\insrt_{i_1}(m_j)$ and $\insrt_{j_1}(m_j)$ are equivalent. Therefore, the matchings $\insrt_{i_1}(m_i)$ and $\insrt_{j_1}(m_j)$ are equivalent.
\end{proof}

\begin{lemma}
\label{lem:matching is generated by inserts to its core}
    Let $m \in \M_N$ be a matching with $k$ unstable chords. Then there exist indices $i_1, \dots, i_k$ such that
    \[
        m = \insrt_{i_1}\cdots\insrt_{i_k}(\core(m)).
    \]
\end{lemma}

\begin{proof}
    We prove the statement by induction on $k$. If $k=0$ then $\core(m) = m$ and the statement is obvious.

    Assume that a matching $m \in \M_N$ has $k > 0$ unstable chords. Therefore, $m$ has at least one short chord, denoted $(i_1,i_1+1) \in m$. Thus, there exists a matching $m_1 \in \M_{N-2}$ such that $m = \insrt_{i_1}(m_1)$. Notice that $\core(m) = \core(m_1)$, so $m_1$ has $k-1$ unstable chords. Thus, by the induction hypothesis, there exist indices $i_2,\dots,i_k$ such that $m_1 = \insrt_{i_2}\cdots\insrt_{i_k}(\core(m))$. Therefore, $m = \insrt_{i_1}\cdots\insrt_{i_k}(\core(m))$.
\end{proof}

Now we are ready to prove Lemma~\ref{lem:if same core then equivalent}.
\begin{proof}[Proof of Lemma~\ref{lem:if same core then equivalent}]
    Let $m_1, m_2 \in \M_N$ be matchings with $m_0 := \core(m_1) = \core(m_2)$. Thus, $m_1$ and $m_2$ have the same number of unstable chords, denoted $k$. By Lemma~\ref{lem:matching is generated by inserts to its core}, we can write $m_1 = \insrt_{i_1}\cdots\insrt_{i_k}(m_0)$ and $m_2 = \insrt_{j_1}\cdots\insrt_{j_k}(m_0)$ for some $i_1,\dots,i_k$ and $j_1,\dots,j_k$. The statement now follows directly from Lemma~\ref{lem:insert many chords to matching is equivalent}.
\end{proof}

Theorem~\ref{thm:knuth like equivalent iff same core} implies the following result:
\begin{corollary}
\label{cor:closed under knuth like is schur positive}
    If a set $\M \subseteq \M_N$ is closed under Knuth equivalence then it is Schur-positive with respect to $\Short$. Moreover, if $\M$ is a Knuth equivalence class then its generating function is
    \[
        \Q(\M) = s_{N-k,k},
    \]
    where $k$ is the number of unstable chords of some arbitrary matching $m \in \M$.
\end{corollary}

\begin{proof}
    Clearly, a disjoint union of Schur-positive sets is Schur-positive. Therefore, it suffices to establish the statement for equivalence classes. Let $\M \subseteq \M_N$ be an equivalence class. By Theorem~\ref{thm:knuth like equivalent iff same core} there exist $0 \le k \le \frac{N}{2}$, $f \ge 0$ and $m_0 \in \M_{N-2k,f}(\Short=\emptyset)$, such that $\M = \{m \in \M_{N,f} \mid \core(m) = m_0\}$. Recall the transformation $F:m \mapsto (\core(m),T(m))$ discussed in Section~\ref{subsec:bijective proof matchings schur positive}, and consider its restriction to $\M$. We obtain that the restricted transformation
    \[
        \restr{F}{\M}:\M \to \{m_0\} \times \Syt(N-k,k)
    \]
    is a statistic-preserving bijection. Thus, applying Theorem~\ref{thm:criterion schur positive young} completes the proof.
\end{proof}

\subsection{Other constructions of Schur-positive sets}

We may apply Corollary~\ref{cor:closed under knuth like is schur positive} to obtain other Schur-positive sets of matchings.
For example, we can filter $\M_N$ by the isomorphism class of the intersection graph:
\begin{definition}
    Let $m$ be a matching. Its \emph{intersection graph}, denoted $G(m)$, is defined to be the undirected simple graph with the chords of $m$ as its vertices, and with an edge between two vertices if the associated chords of $m$ intersect.
\end{definition}

It can be easily seen that applying an elementary Knuth transformation on a matching preserves its intersection graph up to graph-isomorphism. Therefore, by Corollary~\ref{cor:closed under knuth like is schur positive}:

\begin{corollary}
\label{cor:specific intersection graph schur positive}
    For every $N$ and $f$, the set of matchings $m \in \M_{N,f}$ with a fixed intersection graph up to graph-isomorphism is Schur-positive.
\end{corollary}

For example, fix $N$ and $f$. Then for every $k$, the set of $k$-crossing matchings in $\M_{N,f}$ (i.e.,\ matchings where $k$ is the maximal cardinality of a set of pairwise intersecting chords) is Schur-positive. Equivalently, this is the set of matchings whose intersection graph has a maximal clique with $k$ vertices.

As another example, given a matching $m \in \M_{N,f}$, for every $1 \le i \le N$ denote by $I_i(m)$ the number of chords of $m$ that intersect the chord that contains $i$ (if $i$ is an unmatched vertex define $I_i(m)=0$). From Corollary~\ref{cor:specific intersection graph schur positive} we obtain that for any fixed multiset $S$, the set of matchings $m \in \M_{N,f}$ such that $\{I_i(m) \mid i \in [N]\} = S$ (as multisets) is Schur-positive. For example, the following sets are Schur-positive:
\begin{itemize}
    \item The set of matchings $m \in \M_{N,f}$ with exactly $k$ pairs of intersecting chords, i.e.,\ such that $\frac{1}{2} \sum_i I_i(m) = k$.
    \item The set of matchings $m \in \M_{N,f}$ with exactly $k$ intersecting chords, i.e.,\ such that $\frac{1}{2} |\{i \mid I_i(m)>0 \}| = k$.
    \item The set of matchings $m \in \M_{N,f}$ with $k$ the maximal number of times that a chord intersects other chords, i.e.,\ with $\max_i I_i(m) = k$.
\end{itemize}

\subsection{Pattern avoidance in matchings}
\label{subsec:pattern avoiding matchings}

Sagan and Woo~\cite{problem_find_pattern}, motivated by Elizalde and Roichman~\cite{arc_pattern_avoiding}, posed the problem of determining which sets $\Pi$ of permutations satisfy the property that for all $n$, the set of permutations in $S_n$ that avoid every pattern in $\Pi$ is Schur-positive. This problem has been extensively studied since then~\cite{S3_patterns_list, BloomSagan, my_pattern_theorem}.

An analogous question may be asked about Schur-positivity of pattern-avoiding matchings as well. Extensive research has been conducted on pattern avoidance in perfect matchings by Simion and Schmidt~\cite{simion1985restricted}, Jel\'{\i}nek and Mansour~\cite{jelinek2010matchings}, Bloom and Elizalde~\cite{bloom2013pattern}, and others, leading to multiple definitions in the literature. Fang, Hamaker, and Troyka~\cite{fang2022pattern} explore some of these definitions and provide a comparison. Additionally, various conventions exist for generalizing pattern avoidance to non-perfect matchings \cite{mansour2013partial, fang2022pattern, mcgovern2019closures}. We adopt the definition from McGovern~\cite{mcgovern2019closures}, which directly generalizes the definition for perfect matchings from~\cite{jelinek2010matchings} and~\cite{cervetti2022enumeration}.

\begin{definition}
\label{def:pattern avoidance}
    Let $m_1 \in \M_{N_1}$ and $m_2 \in \M_{N_2}$ for some positive integers $N_1 \le N_2$. We say that $m_2$ \emph{contains the pattern} $m_1$, if there exist indices $1 \le i_1 < \dots < i_{N_1} \le N_2$ such that the following holds:
    \begin{itemize}
        \item For all $1 \le j < j' \le N_1$,
        \[
            (j,j') \in m_1 \Longleftrightarrow (i_j,i_{j'}) \in m_2.
        \]
        \item For all $1 \le j \le N_1$,
        \[
            (j) \in m_1 \Longleftrightarrow (i_j) \in m_2.
        \]
    \end{itemize}
    
    Otherwise, we say that $m_2$ \emph{avoids} $m_1$.

    For two sets of matchings $\M_1 \subseteq \M_{N_1}$ and $\M_2 \subseteq \M_{N_2}$, denote by $\M_2(\M_1)$ the set of matchings in $\M_2$ that avoid every matching in $\M_1$. In addition, denote $\M_2(m) := \M_2(\{m\})$ for a matching $m$.
\end{definition}

The following problem is analogous to the problem of Sagan and Woo~\cite{problem_find_pattern} regarding pattern-avoiding permutations:

\begin{problem}
\label{prob:schur positive pattern avoiding matchings}
    Determine which sets $\M \subseteq \M_N$ of matchings satisfy the property that for all $N',f'$, the pattern-avoiding set $\M_{N',f'}(\M)$ is Schur-positive with respect to $\Short$.
\end{problem}

While a complete solution of Problem~\ref{prob:schur positive pattern avoiding matchings} seems challenging, we are able to solve the problem in the case where $|\M| = 1$.

\begin{proposition}
\label{prop:characterize pattern avoiding schur positive singletons}
    Let $N$ be a nonnegative integer, and let $m \in \M_N$ be a matching. Then the pattern-avoiding set $\M_{N',f'}(m)$ is Schur-positive with respect to $\Short$ for all $N',f'$ if and only if one of the following holds:
    \begin{enumerate}
        \item $\Short(m) = \emptyset$, or
        \item $m=\{(1,2)\}$, the unique perfect matching on two vertices.
    \end{enumerate}
\end{proposition}

In the proof of Proposition~\ref{prop:characterize pattern avoiding schur positive singletons}, we will apply the following result:
\begin{lemma}
\label{lem:avoiding set of short chord free matching is knuth like closed}
    Let $m \in \M_N$ be a short-chord-free matching, and let $N',f'$ be nonnegative integers. Then the pattern-avoiding set $\M_{N',f'}(m)$ is closed under Knuth equivalence.
\end{lemma}
\begin{proof}
    Let $m_1' \in \M_{N',f'}$ be a matching, and let $m_2' \in \M_{N',f'}$ be the result of applying an elementary Knuth transformation on $m_1'$. Assume that $m_1'$ contains the pattern $m$. It suffices to prove that $m_2$ contains $m$ too. By Definition~\ref{def:pattern avoidance}, there exist indices $1 \le i_1 < \dots < i_{N} \le N'$ such that $(j,j') \in m \Longleftrightarrow (i_j,i_{j'}) \in m_1'$ for all $1 \le j < j' \le N$ and $(j) \in m \Longleftrightarrow (i_j) \in m_1'$ for all $1 \le j \le N$. Without loss of generality, we may assume that $(i,i+1) \in m_1'$ and $m_2'$ is obtained from $m_1'$ by interchanging the chord $(i,i+1)$ with the vertex $i+2$. The pattern $m$ is short-chord-free, so $i, i+1 \notin \{i_1,\dots,i_N\}$. Therefore,
    by considering the index set $\{i_1,\dots,i_N\}$ if $i+2 \notin \{i_1,\dots,i_N\}$ and the set $\{i\}\cup\{i_1,\dots,i_N\} \setminus \{i+2\}$ otherwise, we obtain that the matching $m_2'$ contains $m$ as well.
\end{proof}

Now, let us prove Proposition~\ref{prop:characterize pattern avoiding schur positive singletons}:
\begin{proof}[Proof of Proposition~\ref{prop:characterize pattern avoiding schur positive singletons}]
    Denote by $f$ the number of unmatched vertices of $m$, and assume that $\M_{N',f'}(m)$ is Schur-positive for all $N',f'$. In particular, we may take $N'=N$ and $f'=f$ and obtain that the set $\M_{N,f}(m)$ is Schur-positive. Since $\M_{N,f}(m) = \M_{N,f} \setminus \{m\}$, we obtain that the generating function $\Q(\M_{N,f} \setminus \{m\})$ is Schur-positive and, as such, symmetric. Clearly,
    \[
        \Q(\M_{N,f} \setminus \{m\})=\Q(\M_{N,f})-\Q(\{m\}).
    \]
    By Theorem~\ref{thm:matchings schur positive}, the function $\Q(\M_{N,f})$ is symmetric, so the function $\Q(\{m\}) = F_{\Short(m)}$ is symmetric as well. By Lemma~\ref{lem:singleton not symmetric}, we may deduce that $\Short(m) = \emptyset$ or $\Short(m) = [N-1]$.
    If $\Short(m) = \emptyset$ then the statement holds. Otherwise, $\Short(m) = [N-1]$. The statistic $\Short$ is sparse (recall Definition~\ref{def:sparse}), so we may deduce that $N\le 2$. If $N=2$ then $\Short(m)=\{1\}$, and therefore $m=\{(1,2)\}$. If $N=1$ then $\Short(m)= \emptyset$.

    On the other hand, we assume that $\Short(m) = \emptyset$ or $m=\{(1,2)\}$ and prove that $\M_{N',f'}(m)$ is Schur-positive for all $N',f'$.
    \begin{enumerate}
        \item Assume that $\Short(m) = \emptyset$, and let $N',f'$ be nonnegative integers.
        By Lemma~\ref{lem:avoiding set of short chord free matching is knuth like closed}, the set $\M_{N',f'}(m)$ is closed under Knuth equivalence. Therefore, by Corollary~\ref{cor:closed under knuth like is schur positive}, it is Schur-positive.
        \item Assume that $m=\{(1,2)\}$, and let $N',f'$ be nonnegative integers. Notice that a matching avoids the pattern $m$ if and only if it is an empty matching (i.e., with all its vertices unmatched). By Definition~\ref{def:knuth like equivalence}, such a matching is Knuth equivalent only to itself, so $\M_{N',f'}(m)$ is closed under Knuth equivalence. Therefore by Corollary~\ref{cor:closed under knuth like is schur positive}, $\M_{N',f'}(m)$ is Schur-positive.
        \qedhere
    \end{enumerate}
\end{proof}

\section{Further remarks and open problems}
\label{sec:further research}

\subsection{Criteria for special cases of Schur-positivity}
There is a simple well-known criterion for sets under which the generating function $\Q_{D}(\A)$ is non-negatively spanned by the Schur functions of hook shape $s_{N-k,1^k}$:
\begin{proposition}[folklore]
\label{prop:hook schur positive}
    A set $\A$ with a statistic $D:\A \to 2^{[N-1]}$ has a generating function of the form
    \[
        \Q_{D}(\A) = \sum_{k=0}^{N-1} c_k s_{N-k,1^k},\ c_k \in \NN
    \]
    if and only if for every $J \subseteq [N-1]$, $|\{a \in \A \mid D(a) = J\}|$ depends only on $|J|$. In that case, $c_k = |\{a \in \A \mid D(a) = [k]\}|$.
\end{proposition}

\begin{proof}
    On one hand, assume that there exist $c_k \in \NN$ such that $|\A(D=J)| = c_{|J|}$ for all $J \subseteq [N-1]$. Notice that every tableau of a hook shape $T \in \Syt(N-k,1^k)$ has $|\Des(T)| = k$. Conversely, for every set $J \subseteq [N-1]$ with $|J| = k$ there exists a unique tableau $T \in \Syt(N-k,1^k)$ such that $\Des(T) = J$. Therefore, we obtain the equation
    \[
        \sum_{a \in \A} \boldsymbol{t}^{D(a)} = \sum_{k=0}^{N-1} c_k \sum_{T \in \Syt(N-k,1^k)} \boldsymbol{t}^{\Des(T)},
    \]
    where $\boldsymbol{t}^J := \prod_{j \in J} t_j$ for $J \subseteq [N-1]$.
    By Theorem~\ref{thm:criterion schur positive young}, this equation implies the statement.

    The other direction of the proposition states that if $\Q_{D}(\A) = \sum_{k=0}^{N-1} c_k s_{N-k,1^k}$, then $|\{a \in \A \mid D(a) = J\}| = c_{|J|}$. Adin and Roichman \cite[Corollary 8.1]{adin_roichman_fine_set} established that a generating function uniquely determines the descent set distribution, which implies the statement.
\end{proof}

Similarly to Proposition~\ref{prop:hook schur positive}, Theorem~\ref{thm:criterion schur positive two rows} provides a simple criterion under which $\Q_{D}(\A)$ is non-negatively spanned by the Schur functions of two-row shape $\{s_{N-k,k} \mid 0 \le k \le n\}$.
In the spirit of discovering new methods and expanding our understanding, we pose the problem of finding other such criteria, hopeful that they will contribute to further advancements in the study of Schur-positivity and Schur coefficients.
\begin{problem}
    Find other sets of partitions $\Lambda \subseteq \{\lambda \vdash N\}$ and criteria for a given set $\A$ under which $\Q_{D}(\A)$ is non-negatively spanned by the functions $\{s_\lambda \mid \lambda \in \Lambda\}$.
\end{problem}

\subsection{Short chords of matchings}
Theorem~\ref{thm:matchings schur positive} shows that the set $\M_{N,f}$ of matchings with a given number of unmatched vertices is Schur-positive with respect to the set of short chords, and describes its Schur expansion. In Section~\ref{sec:refinements}, several Schur-positive subsets of $\M_{N,f}$ are presented too. All these sets are shown to be closed under Knuth equivalence (described in Definition~\ref{def:knuth like equivalence}), so their Schur-positivity can be derived from Corollary~\ref{cor:closed under knuth like is schur positive}. It is desired to find Schur-positive sets of other types.
\begin{problem}
    Find other Schur-positive subsets of $\M_{N,f}$ with respect to $\Short$. In particular, find Schur-positive subsets that are not closed under Knuth equivalence.
\end{problem}

Schur-positive sets of particular interest are pattern-avoiding sets. As discussed earlier in Section~\ref{subsec:pattern avoiding matchings}, Schur-positivity of pattern-avoiding sets was extensively studied in the context of permutations. Recall the problem stated in Section~\ref{subsec:pattern avoiding matchings}:
\begin{problem*}[Problem~\ref{prob:schur positive pattern avoiding matchings} above]
    Determine which sets $\M \subseteq \M_N$ of matchings satisfy the property that for all $N',f'$, the pattern-avoiding set $\M_{N',f'}(\M)$ is Schur-positive with respect to $\Short$.
\end{problem*}

Proposition~\ref{prop:characterize pattern avoiding schur positive singletons} solves this problem for pattern sets of size $1$. We do not know a general answer for larger pattern sets.

The connection identified between the Schur expansion of $\M_{N,f}$ and Bessel polynomials, as revealed in Corollary~\ref{cor:schur expansion bessel}, serves as motivation to investigate the Schur expansion of additional sets:
\begin{problem}[Sergi Elizalde, personal communication]
    Find the Schur expansion of Schur-positive sets of matchings. For instance, find the Schur expansion of the set of $k$-crossing matchings in $\M_{N,f}$, or the Schur expansion of Schur-positive pattern-avoiding sets.
\end{problem}

\subsection{The involutive length}
\label{subsec:schreier graph of perfect matchings}

Here, we discuss one of the primary motivations behind this research. We denote by $\I_{2n}$ the set of fixed-point-free involutions in the symmetric group $S_{2n}$. Furthermore, we define the \emph{simple reflections} $s_i = (i,i+1) \in S_{2n}$ for $1 \le i < 2n$, and denote $w_0 = (1,2)(3,4)\cdots (2n-1,2n)$. Adin, Postnikov, and Roichman~\cite{gelfand_involutions} defined the \emph{involutive length} of an involution $w \in \I_{2n}$ as
\[
    \hat{\ell}(w) := \min\{k \mid s_{i_1}\cdots s_{i_k} w (s_{i_1}\cdots s_{i_k})^{-1} = w_0,\ 1 \le i_1,\dots,i_k <2n\}.
\]
They also defined the \emph{involutive weak order} $\le_\I$ on $\I_{2n}$, as the reflexive and transitive closure of the relation $w \prec_\I s_i w s_i$ if $\hat{\ell}(s_i w s_i) = \hat{\ell}(w) + 1$. The involutive order is motivated by the Bruhat order defined for Coxeter groups \cite[Chapter 2]{bjorner2005combinatorics}.

Since fixed-point-free involutions in $S_{2n}$ naturally correspond to perfect matchings on $2n$ vertices, we will adopt the involutive length and the involutive weak order for perfect matchings in $\PM_{2n} := \M_{2n,0}$ and use the same notation for them.

Consider the natural action of $S_{2n}$ on the set $\PM_{2n}$, and denote $m_0 = \{(1,2),\; (3,4),\dots,\;\allowbreak (2n-1,2n)\}$. Notably, the stabilizer of $m_0$ is isomorphic to the hyperoctahedral group $H_n = S_2 \wr S_n$. Furthermore, consider the Schreier graph associated with this action with respect to the simple reflections. 
For an illustration of this Schreier graph, see Figure~\ref{fig:schreier graph example} which depicts it when $2n=4$.


\begin{figure}[htb]
    \begin{center}
        \usetikzlibrary{shapes}
        \begin{tikzpicture}[node distance={15mm}, thick, main/.style = {draw, ellipse, scale=1}]  
            \node[main] (1) {$12|34$};  
            \node[main] (2) [above of=1] {$13|24$}; 
            \node[main] (3) [above of=2] {$14|23$};
            
            \draw (1) to (2) node[below=23pt,right, scale=1] {$2$}; 
            \draw (2) to[looseness=1.5, in=-50, out=50] (3) node[below=23pt,right=15pt, scale=1] {$3$};  
            \draw (2) to[looseness=1.5, in=-130, out=130] (3) node[below=23pt,left=15pt, scale=1] {$1$};  
            \draw (1) to [out=160,in=200,looseness=5] (1) node[below=0pt,left=39pt, scale=1] {$1$};  
            \draw (1) to [out=20,in=-20,looseness=5] (1) node[below=0pt,right=39pt, scale=1] {$3$};  
            \draw (3) to [out=160,in=200,looseness=5] (3) node[below=0pt,left=39pt, scale=1] {$2$};  
        \end{tikzpicture} 
    \end{center}
    \caption{Schreier graph of the natural action of $S_4$ on $\PM_4$}
    \label{fig:schreier graph example}
\end{figure}



We can create a layered graph by partitioning the vertices (matchings) into layers based on their distance from $m_0$. Specifically, a matching in the $\ell$-th layer is at a distance of $\ell$ from $m_0$. Furthermore, Avni~\cite{avni_perfect_matchings} proved that this graph is bipartite when ignoring loops. Therefore, a matching in one layer is connected to matchings in the adjacent layers, but not within its own layer or to matchings in non-adjacent layers.

It turns out that this graph corresponds to the involutive length and the involutive weak order. First, the index of the layer of a matching $m$ is equal to its involutive length $\hat{\ell}(m)$. Moreover, for every $m_1, m_2 \in \PM_{2n}$ we have $m_1 \le_\I m_2$ if and only if there exists a geodesic path from $m_0$ to $m_2$ passing through $m_1$.

Based on the structure of the Schreier graph, we define three natural set-valued functions on perfect matchings, denoted by $\Asc(m)$, $\Loop(m)$, and $\Des(m)$, where $m$ is a matching in $\PM_{2n}$. Given a matching $m \in \PM_{2n}$ and an index $1 \leq i \leq 2n-1$, we define:
\begin{itemize}
    \item $i \in \Asc(m)$ if $s_i \cdot m >_\I m$, where the dot denotes the group action,
    \item $i \in \Des(m)$ if $s_i \cdot m <_\I m$, and
    \item $i \in \Loop(m)$ if $s_i \cdot m = m$.
\end{itemize}
It is noteworthy that $i\in \Asc(m)$ if and only if $\hat{\ell}(s_i \cdot m) > \hat{\ell}(m)$, $i\in \Des(m)$ if and only if $\hat{\ell}(s_i \cdot m) < \hat{\ell}(m)$, and $i\in \Loop(m)$ if and only if $\hat{\ell}(s_i \cdot m) = \hat{\ell}(m)$.


The algebraic interest of the graph motivates the analysis of these three statistics, specifically focusing on the question of the Schur-positivity of $\PM_{2n}$ with respect to each of them. As a first step in this direction, it is noteworthy that the definition of $\Asc$ coincides with the standard ascents of permutations when restricted to $\I_{2n}$. Consequently, the Schur-positivity of $\PM_{2n}$ with respect to $\Asc$ can be derived directly from classical properties of the Robinson-Schensted correspondence, as noted for example by Gessel and Reutenauer \cite[end of Section 7]{gessel1993counting}.
Moreover, we have the observation $\Loop(m) = \Short(m)$ for a perfect matching $m$. This observation, when combined with Theorem~\ref{thm:matchings schur positive}, implies that $\PM_{2n}$ is Schur-positive with respect to $\Loop$.

In contrast, the Schur-positivity of $\PM_{2n}$ with respect to $\Des$ remains a mystery.

\begin{question}[Ron Adin and Yuval Roichman, personal communication]
    Is $\PM_{2n}$ Schur-positive with respect to $\Des$?
\end{question}

Despite extensive analysis of the $\Des$ statistic, we are currently unable to provide an answer. However, based on experimental results from simulations, we have discovered that $\PM_{2n}$ is Schur-positive with respect to $\Des$ when $2n \leq 14$. Moreover, intriguingly, the statistics $\Asc$ and $\Des$ are found to be equidistributed for all $2n \leq 14$. These findings lead us to propose the following conjecture:

\begin{conjecture}[Ron Adin and Yuval Roichman, personal communication]
    For all $n \in \NN$,
    \[
        \sum_{m \in \PM_{2n}} \boldsymbol{t}^{\Asc(m)} = \sum_{m \in \PM_{2n}} \boldsymbol{t}^{\Des(m)},
    \]
    where $\boldsymbol{t}^J := \prod_{j \in J} t_j$ for $J \subseteq [2n-1]$.
\end{conjecture}

If this conjecture can be proven, it would establish the Schur-positivity of $\PM_{2n}$ with respect to $\Des$. Furthermore, a bijection on $\PM_{2n}$ that maps ascents to descents may unveil hidden symmetries within the Schreier graph.


\section{Acknowledgements}
My sincere thanks to my supervisors, Ron Adin and Yuval Roichman, for their constant support. I also wish to thank Yotam Shomroni, Noam Ta-Shma and Ohad Sheinfeld, for enlightening discussions and valuable editing suggestions. Special appreciation goes to Sergi Elizalde and Bruce Sagan for their insightful comments and suggestions, which significantly contributed to refining the paper.

\printbibliography
\end{document}